\documentclass[13pt, a4paper]{article}
 \usepackage{amsmath}
 \usepackage{amssymb}
\usepackage[margin=21mm]{geometry}
\usepackage[utf8]{inputenc}
\usepackage{graphics}
\usepackage{xspace}
\usepackage{color}
 \usepackage{amsthm}
 \usepackage[old]{old-arrows}
  \usepackage{amsfonts}
\usepackage[dvipsnames]{xcolor}
\usepackage{mathtools}
\usepackage{amssymb}
\usepackage{latexsym}
\newcommand{\R}{{\mathbb R}}
\newcommand{\ren}{\R^{N}}
\newcommand{\irn}{\int\limits_{\re^N}}
\newcommand{\N}{{\mathbb N}}
\newcommand{\Z}{{\mathbb Z}}
\renewcommand{\l }{\lambda }
\newcommand{\loc}{\operatorname{loc}}
\newcommand{\p}{\partial}
\renewcommand{\O}{\Omega }
\newcommand{\g}{\gamma }
\newcommand{\bra}{\langle}
\newcommand{\ket}{\rangle}
\newcommand{\iy}{\infty}
\newcommand{\SN}{{\mathbb S}^{N-1}}
\newcommand{\onn}{\text{  on   }}
\newcommand{\inn}{\text{  in   }}
\newcommand{\dyle}{\displaystyle}
\newcommand{\dint}{\dyle\int}
\newcommand{\weakly}{\rightharpoonup}
\newcommand{\e }{\varepsilon}
\newcommand{\dive }{\mathop{\rm div}}
\newcommand{\sgn}{\mathop{\rm sgn}}
\newcommand{\h}{{\mathbb H}}
\usepackage[mathscr]{euscript}
\usepackage{mathrsfs}
 \usepackage{enumerate}
 \usepackage{cite}
 \usepackage{tikz}
 \usepackage[pagebackref]{hyperref}
 \usepackage{hyperref}

\newcommand{\eps}{\varepsilon}
\catcode`@=12
\def\eqdef{{\buildrel \rm def \over =}}
\def\eop{{\ \vrule height 7pt width 7pt depth 0pt}}
\makeatother
\usepackage[english]{babel}
\newcommand{\D}{\mathbb{D}}
\newcommand{\Q}{\mathbb{Q}}
 \usepackage[hyperpageref]{backref}
\newtheorem{theorem}{Theorem}
\newtheorem{remark}[theorem]{Remark}
\newtheorem{lemma}[theorem]{Lemma}
\newtheorem{proposition}[theorem]{Proposition}
\newtheorem{corollary}[theorem]{Corollary}

\DeclareMathOperator*{\supp}{\text{supp}}

\newcommand{\ph}{\varphi}
\newcommand{\into}{\int_{\Omega}}
\newcommand{\intr}{\iint_{\R^N\times\R^N}}

\renewcommand{\l}{\left}
\renewcommand{\r}{\right}

\def\abs#1{\left|{#1}\right|}

\numberwithin{theorem}{section}
\numberwithin{equation}{section}

\title{On the Fu\v{c}\'{i}k spectrum of the Logarithmic Laplacian}
\date{}
\author{Rakesh Arora$^{1}$\footnote{ R. Arora, \textit{E-mail address:}
  \texttt{rakesh.mat@iitbhu.ac.in}} \ and Tuhina Mukherjee$^{2} \footnote{T.~Mukherjee, \textit{E-mail address:} \texttt{tuhina@iitj.ac.in}}$, \\
       \small $^{1}$ Department of Mathematical Sciences, Indian Institute of Technology (IIT-BHU) Varanasi, Uttar Pradesh 221005, India\\
\small $^{2}$ Department of Mathematics, Indian Institute of Technology Jodhpur, Rajasthan 342030, India \\}
\newcommand{\Addresses}{{ additional braces for segregating \footnotesize
  \footnotesize
  R. Arora, \textit{E-mail address:}
  \texttt{rakesh.mat@iitbhu.ac.in}\\
  \medskip
  T.~Mukherjee, \textit{E-mail address:} \texttt{tuhina@iitj.ac.in}
}}
\providecommand{\keywords}[1]
{
  \small	
  \textbf{\textit{Keywords---}} #1
}
\begin{document}
\maketitle \vspace{-1.8\baselineskip}
\begin{abstract}
 In this paper, we investigate the Fu\v{c}\'{i}k spectrum $\Sigma_L$ associated with the logarithmic Laplacian. This spectrum is defined as the set of all pairs $(\alpha,\beta) \in \mathbb{R}^2$ for which the problem
\begin{equation*}
    \left\{\begin{aligned}
   L_\Delta u\: &= \alpha u^+-\beta u^- &&~~\text{in} ~~ \Omega, \\
      u&=0 &&~~\text{in} ~~\mathbb R^N\setminus \Omega,
    \end{aligned} \right.
\end{equation*}
admits a nontrivial solution $u$. Here, $\Omega \subset \mathbb{R}^N$ is a bounded domain with $C^{1,1}$ boundary, $u^\pm = \max\{\pm u,0\}$, and $u = u^+ - u^-$. We show that the lines $\lambda_1^L \times \mathbb{R}$ and $\mathbb{R} \times \lambda_1^L$, where $\lambda_1^L$ denotes the first eigenvalue of $L_\Delta$, lies in the spectrum $\Sigma_L$ and are isolated within the spectrum. Furthermore, we establish the existence of the first nontrivial curve in $\Sigma_L$ and analyze its qualitative properties, including Lipschitz continuity, strict monotonicity, and asymptotic behavior. In addition, we obtain a variational characterization of the second eigenvalue of the logarithmic Laplacian and show that all eigenfunctions corresponding to eigenvalues $\lambda > \lambda_1^L$ are sign-changing. Finally, we address a nonresonance problem with respect to the Fu\v{c}\'{i}k spectrum $\Sigma_L$, employing variational methods and carefully overcoming the difficulties arising from the contrasting features of the first eigenvalue $\lambda_1^L$.
\end{abstract}

\\

 
\section{Introduction}\label{main-results}
In this work, we study the Fu\v{c}\'{i}k spectrum of the logarithmic Laplacian, which is defined as the set of all pairs $(\alpha, \beta) \in \mathbb{R}^2$ for which the problem 
\begin{equation*} \label{1}
    \left\{\begin{aligned}
   L_\Delta u\: &= \alpha u^+-\beta u^- &&~~\text{in} ~~ \Omega, \\
      u&=0 &&~~\text{in} ~~\mathbb R^N\setminus \Omega,
    \end{aligned} \right. \tag{$P_{\alpha, \beta}$}
\end{equation*}
admits a nontrivial solution $u.$ Here, $u^+ = \max\{u,0\}$, $u=u^+-u^-$, $\Omega$ is a bounded domain of $\mathbb R^N$ with $\partial \Omega \in C^{1,1}$ and the operator $L_\Delta$ denotes the logarithmic Laplacian, {\it i.e.,} the pseudo-differential operator with Fourier symbol $2 \ln |\xi|,$ which can also be seen as a first-order expansion of the fractional Laplacian (the pseudo-differential operator with Fourier symbol $|\xi|^{2s}$). In particular, for $u \in C_c^2(\mathbb{R}^N)$ and $x \in \mathbb{R}^N,$
\begin{equation}\label{first-order-exp}
    (-\Delta)^s u(x) = u(x) + s L_{\Delta} u(x) + o(s) \quad s \to 0^+ \quad \text{in} \ L^p(\mathbb{R}^N), \ 1 < p \leq \infty.
\end{equation}  
Recall that for $s \in (0,1)$, the fractional Laplacian $(-\Delta)^s$ can be written as a singular integral operator defined in the principal value sense
\[
\begin{split}
(-\Delta)^su(x) = c(N,s) \ \text{P.V.} \int_{\R^N} \frac{u(x)-u(y)}{\abs{x-y}^{N+2s}} ~dx,
\end{split}
\] \\
where $c(N,s) =2^{2s} \pi^\frac{-N}{2} s \frac{\Gamma\l(\frac{N+2s}{2}\r)}{\Gamma(1-s)}$ is a normalizing constant. In the same spirit, the operator $L_\Delta$ has the following integral representation (see \cite[Theorem 1.1]{ChenWeth2019})
\begin{equation}\label{def:log-lap-ope}
    L_\Delta u(x) = c_N \int_{\mathcal{B}_1(x)} \frac{u(x)-u(y)}{\abs{x-y}^N} ~dy -c_N\int_{\R^N\setminus\mathcal{B}_1(x)} \frac{u(y)}{\abs{x-y}^N} ~dy + \rho_Nu(x),
\end{equation}
where $\mathcal{B}_1(x) \subset \mathbb{R}^N$ denotes the Euclidean ball of radius $1$ centered at $x$ and
\[
c_N := \pi ^{\frac{-N}{2}}\Gamma(\frac{N}{2}), \quad \rho_N := 2\ln 2+ \psi(\frac{N}{2})-\gamma, \quad  \gamma := -\Gamma^{'}(1),\]
 $\gamma$ being the Euler-Mascheroni constant and $\psi := \frac{\Gamma^{'}}{\Gamma}$ the digamma function. The singular kernel in \eqref{def:log-lap-ope} is sometimes called of {\it zero-order}, because it is the limiting case of hyper singular integrals. These operators arise naturally both in mathematical models exhibiting borderline singular behavior and in a variety of applied contexts. Their study has produced a rich theoretical framework; see \cite{Arora-Giacomoni-Vaishnavi, AroGiaHajVai-2025, ChenWeth2019, HernandezLopezSaldana2024, Feulefack-Jarohs, Frank-Konig-Tang, Foghem} for
mathematical developments and \cite{Pellacci-Verzini, Sikic-Song-Vondracek, Sprekels-Valdinoci} for applications. The appropriate functional setting for studying the Dirichlet problem involving the logarithmic Laplacian
$L_\Delta$ was recently introduced by Foghem \cite{Foghem}. For $q \in [1,\infty]$, we denote by $L^q(\Omega)$
the usual Lebesgue space endowed with the norm
\[
\|u\|_{q}:=\left(\int_{\Omega} |u|^q \, dx\right)^{\frac{1}{q}}
\quad \text{for } 1 \le q < \infty,
\qquad
\|u\|_{\infty}:=\operatorname*{ess\,sup}_{\Omega} |u|.
\]
We introduce the kernels
\begin{equation}\label{kernels:op}
k, j : \mathbb{R}^N \setminus \{0\} \to \mathbb{R},
\qquad
k(z) = \frac{\mathbf{1}_{B_1}(z)}{|z|^N},
\quad
j(z) = \frac{\mathbf{1}_{\mathbb{R}^N \setminus B_1}(z)}{|z|^N}.
\end{equation}
For problem~\eqref{1}, the natural energy space is defined as (see \cite{ChenWeth2019})
\[
\begin{aligned}
\mathbb{H}(\Omega)
:= \big\{ u \in L^2(\Omega) :\;
& u = 0 \text{ in } \mathbb{R}^N \setminus \Omega, \ \text{and}~\iint_{\mathbb{R}^{2N}} |u(x)-u(y)|^2 k(x-y)\,dx\,dy < \infty
\big\}.
\end{aligned}
\]
The associated inner product and norm on $\mathbb{H}(\Omega)$ are given by
\[
\mathcal{E}(u,v)
:= \frac{c_N}{2}
\iint_{\mathbb{R}^{2N}} (u(x)-u(y))(v(x)-v(y)) k(x-y)\,dx\,dy,
\qquad
\|u\| := \mathcal{E}(u,u)^{1/2}.
\]
The quadratic form associated with the operator $L_\Delta$ is
\begin{equation}\label{quadratic-form}
\mathcal{E}_L(u,v)
= \mathcal{E}(u,v)
- c_N \iint_{\mathbb{R}^{2N}} u(x)v(y) j(x-y)\,dx\,dy
+ \rho_N \int_{\mathbb{R}^N} u v \, dx.
\end{equation}
It follows from \cite[Theorem~2.1]{CorreaDePablo2018} and
\cite[Corollary~2.3]{LaptevWeth2021} that the embedding
\begin{equation}\label{compact-embed}
\mathbb{H}(\Omega) \hookrightarrow L^2(\Omega)
\quad \text{is compact}.
\end{equation}
More recently, Arora, Giacomoni, and Vaishnavi \cite[Theorem~2.2]{Arora-Giacomoni-Vaishnavi} established optimal continuous embeddings of $\mathbb{H}(\Omega)$ into Orlicz-type spaces $L^\varphi(\Omega)$, where $\varphi(t) \approx t^2 \log (e+t)$ for $t \gg 1$, and compact embeddings into Orlicz spaces $L^\psi(\Omega)$ satisfying $\lim_{t \to \infty} \frac{\psi(t)}{\varphi(t)} = 0.$

The study of the Fu\v{c}\'{i}k spectrum has attracted sustained interest in nonlinear analysis and spectral theory over the past several decades. The pioneering work of Fu\v{c}\'{i}k \cite{Fucik1964} initiated the investigation of nonlinear eigenvalue problems involving asymmetric jumping nonlinearities, thereby introducing the notion of the Fu\v{c}\'{i}k spectrum. Since then, this concept has served as a fundamental framework for the analysis of a wide class of nonlinear boundary value problems.

Building on this foundation, Cuesta, de Figueiredo, and Gossez \cite{cuesta} extended the theory to the $p$-Laplacian operator, providing a detailed study of the structure and qualitative properties of the Fu\v{c}\'{i}k spectrum in a nonlinear setting. Their work yielded significant insights into bifurcation phenomena and the multiplicity of solutions. For further developments concerning the Laplacian and the $p$-Laplacian, we refer the reader to \cite{gossez2, Perera-2002, Perera-2004} and the references therein.

In the context of nonlocal operators, Goyal and Sreenadh \cite{GoyalSreenadh2014} investigated the Fu\v{c}\'{i}k spectrum associated with the fractional Laplacian, extending several classical results from local operators to the nonlocal framework. Using variational and topological methods, they established the existence of a first nontrivial curve in the Fu\v{c}\'{i}k spectrum and analyzed its monotonicity and asymptotic behaviour. Earlier, Sreenadh \cite{Sreenadh2002} studied the Fu\v{c}\'{i}k spectrum of the Hardy--Sobolev operator, introducing weighted Sobolev spaces to handle nonlinear eigenvalue problems involving singular potentials.

More recently, Goel, Goyal, and Sreenadh \cite{GoelGoyalSreenadh2018} examined the Fu\v{c}\'{i}k spectrum of the $p$-fractional Laplacian subject to nonlocal normal derivative conditions. We also mention the recent work \cite{GoyalSreenadhSteklov2025}, which addresses the Fu\v{c}\'{i}k spectrum for mixed local-nonlocal operators, further enriching the theory and highlighting the growing interest in nonlocal and mixed frameworks.

The Fu\v{c}\'{i}k spectrum associated with problem \eqref{1} naturally extends the classical spectral theory of the logarithmic Laplacian. Indeed, when the parameters satisfy $\alpha=\beta$, problem \eqref{1} reduces to the standard eigenvalue problem
\[
\left\{
\begin{aligned}
L_\Delta u &= \lambda u && \text{in } \Omega, \\
u &= 0 && \text{in } \mathbb{R}^N \setminus \Omega.
\end{aligned}
\right.
\]

In \cite{ChenWeth2019}, Chen and Weth established the existence of an increasing and unbounded sequence of eigenvalues
\[
\lambda_1^L < \lambda_2^L \le \cdots \le \lambda_k^L \le \cdots,
\qquad
\lim_{k \to \infty} \lambda_k^L = +\infty.
\]

The first eigenvalue $\lambda_1^L$ is simple and admits eigenfunctions of constant sign; however, unlike the classical Laplacian and the fractional Laplacian, it may be negative. Moreover, $\lambda_1^L$ has a variational characterization,
\[
\lambda_1^L
= \inf_{u \in \mathbb{H}(\Omega)}
\bigl\{ \mathcal{E}_L(u,u) : \|u\|_2 = 1 \bigr\}.
\]

For every $k \in \mathbb{N}$, the point $(\lambda_k^L,\lambda_k^L)$ belongs to the Fu\v{c}\'{i}k spectrum $\Sigma_L$. Yet, beyond these diagonal points, the geometry of $\Sigma_L$ has remained unexplored. To the best of our knowledge, a systematic study of the Fu\v{c}\'{i}k spectrum associated with problem \eqref{1} has not been carried out so far, and uncovering its structure is precisely the objective of the present work.

Although our work is motivated by the methodologies developed in \cite{cuesta, GoyalSreenadh2014} for the local $p$-Laplacian and the nonlocal fractional Laplacian, respectively, the intrinsically nonlocal nature of the operator $L_\Delta$, together with the contrasting properties of its first eigenvalue $\lambda_1^L$, gives rise to significant analytical challenges that prevent a straightforward extension of the techniques employed in those works. The main challenges and contributions of this paper are summarized as follows. 

\begin{itemize}
    \item First, we study the sign-changing properties of the solution of the weighted logarithmic Laplacian problem with bounded weights away from $\lambda_1^L$ on a subset of positive measure, see Theorem \ref{thm-sign-changing}. These properties play a crucial role in both the identification of isolated points in the spectrum and the construction of nontrivial curves in the Fu\v{c}\'{i}k spectrum of $L_\Delta$. As an application of this, in Corollary \ref{pos-first-ef}, we also show that the eigenfunctions corresponding to eigenvalues $\lambda > \lambda_1^L$ are sign-changing, which is of independent interest.
    
    \item Next, we show that the lines $\lambda_1^L \times \mathbb{R}$ and $\mathbb{R} \times \lambda_1^L$, where $\lambda_1^L$ denotes the first eigenvalue of $L_\Delta$, belong to the spectrum $\Sigma_L$. To this end, we introduce a constrained energy functional $\tilde{E}_r$ (see Section~\ref{section:curve-construction}) and establish the existence of its critical points, which is equivalent to identifying points in the spectrum $\Sigma_L$. The existence of the first critical point, $\varphi_1$, an eigenfunction associated with $\lambda_1^L$, is obtained via global minimization of $\tilde{E}_r$ on a suitable constraint set. This yields the vertical line $\lambda_1^L \times \mathbb{R}$ belong to the spectrum $\Sigma_L$, see Proposition \ref{prop2}. The existence of a second critical point, $-\varphi_1$, follows from a local minimization argument and gives rise to the horizontal line $\mathbb{R} \times \lambda_1^L$, see Proposition \ref{prop3}. Finally, the existence of a third critical point is established by verifying the Palais--Smale condition and the mountain pass geometry of the constrained functional in a neighbourhood of the local minimum $-\varphi_1$. Applying the Mountain Pass Theorem, in Theorem \ref{third-crit-pt}, we obtain a spectral curve of the form $(r + c(r),\, c(r))$ and $(c(r),\, r + c(r))$ for $r \ge 0$, where $c(r)$ denotes the corresponding critical value of the constrained energy functional $\tilde{E}_r$.
    
    \item In Theorem~\ref{isolated}, we prove that the spectral lines $\lambda_1^L \times \mathbb{R}$ and $\mathbb{R} \times \lambda_1^L$ obtained above are isolated in $\Sigma_L$. In the works \cite{cuesta, GoyalSreenadh2014} dealing with the local $p$-Laplacian and the nonlocal fractional Laplacian, this property is established using continuous embeddings of the underlying energy space into Lebesgue spaces $L^r(\Omega)$ with $r>2$. However, such embeddings are not available for the energy space $\mathbb{H}(\Omega)$ associated with the operator $L_\Delta$. To overcome this obstruction, we employ analytical techniques from Orlicz space theory and make essential use of the recent optimal embedding results into Orlicz-type spaces established in \cite[Theorem~2.2]{Arora-Giacomoni-Vaishnavi}. 

    \item In Theorem~\ref{mainthrm-1}, we show that the point $(r + c(r),\, c(r))$ constitutes the first nontrivial element of $\Sigma_L$ lying on the line parallel to the diagonal passing through $(r,0)$. The proof relies crucially on the construction of a continuous path connecting the first and second critical points, $\varphi_1$ and $-\varphi_1$, of the constrained energy functional $\tilde{E}_r$, which remains strictly below the critical level $c(r)$. In the construction of such path, the nonlocal nature of the operator, combined with the presence of sign-changing terms in the associated quadratic form, gives rise to several additional tail contributions that are absent in the classical local and fractional settings. Controlling these terms requires a delicate analysis and constitutes a substantial technical difficulty in the argument. As a further consequence of Theorem~\ref{mainthrm-1}, we obtain a variational characterization of the second eigenvalue $\lambda_2^L$, which is new in the existing literature.

    \item Next, we investigate the monotonicity, regularity, and asymptotic behaviour of the first nontrivial spectral curve; see Propositions~\ref{curve-prop-1}, \ref{curve-prop-2}, and \ref{curve-prop-3}. In the works \cite{cuesta, GoyalSreenadh2014} concerning the local $p$-Laplacian and the nonlocal fractional Laplacian, these properties are established using the results based on the positivity of the principal eigenvalue. In contrast, for the operator $L_\Delta$, the first eigenvalue may change sign, which introduces a substantial analytical difficulty. This obstacle is overcome by exploiting the sign-changing properties of solutions to the weighted logarithmic Laplacian problem with weights bounded away from $\lambda_1^L$ on a subset of positive measure, as established in Theorem~\ref{thm-sign-changing}, together with its consequences in Lemma~\ref{lem-perturbed-fucik}. Combining the properties established above, Figure~\ref{ht} illustrates the diagonal points and the first nontrivial curve in the Fu\v{c}\'{i}k spectrum associated with $L_\Delta$.

        \begin{figure}[ht]\label{ht}
 \centering
 \begin{tikzpicture}[scale=0.8]

  Axes
 \draw[->] (-2.2,0) -- (8,0) node[right] {$\alpha$};
 \draw[->] (0,-2.2) -- (0,7) node[above] {$\beta$};

  Isolated eigenvalue lines
 \draw[very thick] (2,-2) -- (2,8);
 \draw[very thick] (-2,2) -- (8,2);

 \node[below] at (1.7,0) {$\lambda_1^L$};
 \node[left] at (0,1.7) {$\lambda_1^L$};

 \node[below right] at (2,-1.2) {$\lambda_1^L \times \mathbb{R}$};
 \node[above left] at (-1,2) {$\mathbb{R} \times \lambda_1^L$};

  Second eigenvalue guides
 \draw[dashed] (4,0) -- (4,4);
 \draw[dashed] (0,4) -- (4,4);

 \node[below] at (4,0) {$\lambda_2^L$};
 \node[left] at (0,4) {$\lambda_2^L$};

  Diagonal reference
 \draw[densely dashed] (0,0) -- (7.5,7.5);

  Diagonal eigenvalue points
 \fill (2,2) circle (2.2pt);
 \node[below left] at (2,2.7) {$(\lambda_1^L,\lambda_1^L)$};

 \fill (4,4) circle (2.2pt);
 \node[above right] at (4,3.7) {$(\lambda_2^L,\lambda_2^L)$};

 \fill (7,7) circle (2.2pt);
 \node[above right] at (7,6.5) {$(\lambda_k^L,\lambda_k^L)$};

  First nontrivial Fu\v{c}\'{i}k curve
 \draw[very thick, domain=2.7:9, smooth, variable=\x]
 plot ({\x},{2 + 4/(\x-2)});

  Annotation
 \node at (7.6,2.32) {First nontrivial curve};

   Marked points on the curve
  \fill (3.1,5.7) circle (2pt);
  \node[left] at (6.1, 6.2) {$(c(r),\, r+c(r))$};

  \fill (5.7,3.09) circle (2pt);
  \node[right] at (5.7,3.4) {$(r+c(r),\, c(r))$};

 \end{tikzpicture}
 \caption{Structure of the Fu\v{c}\'{i}k spectrum $\Sigma_L$ (thick lines, a non-trivial curve and diagonal points)}
 \label{fig:fucik-spectrum}
 \end{figure}
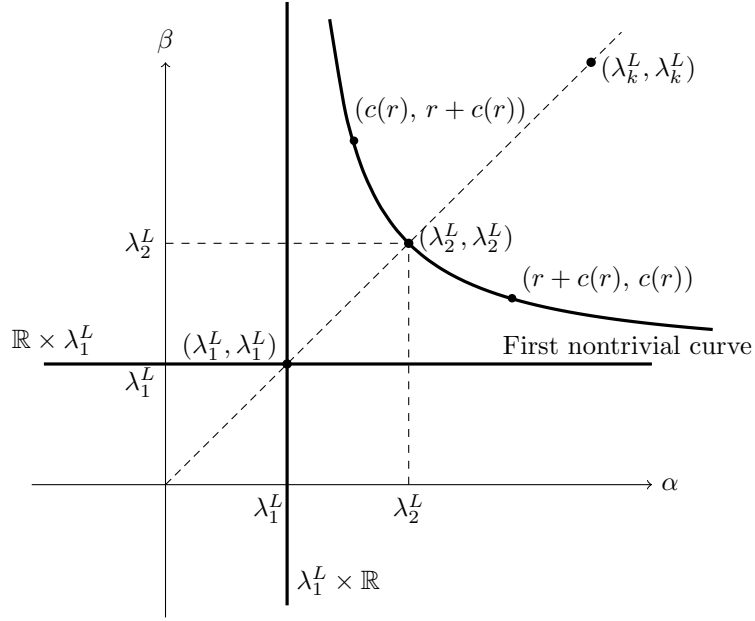

 \item Finally, we address a nonresonance problem involving the logarithmic Laplacian with asymptotic linearities lying between $(\lambda_1^L, \lambda_1^L)$ and the first nontrivial spectral curve. By employing variational methods and exploiting the consequences of Theorem~\ref{thm-sign-changing}, we establish the existence of a nontrivial solution.
\end{itemize}

\textbf{Outline of the paper:} In Section~\ref{section-sign-chang}, we derive the sign-changing properties of the solution of the weighted logarithmic Laplacian problem with bounded weights away from $\lambda_1$ on a subset of positive measure. In Section~\ref{section:curve-construction}, we introduce a parametrized constrained functional $\tilde E_r$, $r \in \mathbb{R}$, whose critical points characterize the Fu\v{c}\'{i}k spectrum $\Sigma_L$. Within this framework, we show that the lines $\lambda_1^L \times \mathbb{R}$ and $\mathbb{R} \times \lambda_1^L$, where $\lambda_1^L$ is the first eigenvalue of $L_\Delta$, belong to $\Sigma_L$. By applying the Mountain Pass Theorem, we further obtain a spectral curve of the form $(r+c(r),c(r))$ and $(c(r),r+c(r))$ for $r \ge 0$. Section~\ref{section-nontrivial-curve} shows that the above lines are isolated in $\Sigma_L$, that the obtained curve is nontrivial, and provides a variational characterization of the second eigenvalue $\lambda_2^L$. In Section~\ref{section-curve-prop}, we establish the strict monotonicity, Lipschitz continuity, and asymptotic behavior of this curve. Finally, in Section~\ref{section:nonresonance}, we study a nonresonance problem with respect to $\Sigma_L$ and prove the existence of a nontrivial solution.

\section{Sign-changing properties}\label{section-sign-chang}
We begin by showing the geodesic convexity property of the energy term $\mathcal{E}_L$ and the sign-changing properties of the solution of the weighted logarithmic Laplacian problem with bounded weights away from $\lambda_1$ on a subset of positive measure. As an application of this, we show that the eigenfunctions corresponding to eigenvalues $\lambda > \lambda_1$ are sign-changing. 

\begin{lemma}\label{lem:convexity}
    Let $N \geq 1$ and $W: \h(\Omega) \to \mathbb{R}$ be the functional defined as
    \begin{equation}\label{log-laplace:functional}
    W(u):= \mathcal{E}_L(u,u).
\end{equation}
   For functions $u, v \in \h(\Omega)$, consider the function $\sigma_t$ defined by
    \[
    \sigma_t(x) := \left(t u^2 + (1-t) v^2\right)^\frac{1}{2}, \quad \text{for all} \ t \in [0,1].
    \]
Then, 
\[
W(\sigma_t) \leq t W(u) + (1-t) W(v).
\]
\begin{proof}
Let $u \in \h(\Omega).$ By \cite[proposition 3.2]{ChenWeth2019}, we have
\[
\begin{split}
\mathcal{E}_L(u,u) &= \frac{c_N}{2} \into \into \frac{(u(x)-u(y))^2}{|x-y|^N} ~dx ~dy + \into \l(h_{\Omega}(x) + \rho_N\r) u^2(x) ~dx\\
&:= I_1(u) + I_2(u),
\end{split}
\]
where
\[
h_\Omega(x) = c_N \l(\int_{B_1(x) \setminus \Omega} \frac{1}{|x-y|^N} ~dy - \int_{\Omega \setminus B_1(x)} \frac{1}{|x-y|^N} ~dy\r).
\]
Notice that
\[
\sigma_t \equiv \|(t^\frac{1}{2} u, (1-t)^\frac{1}{2} v)\|_{\ell^2}
\]
where $\|\cdot\|_{\ell^2}$ denotes the $\ell^2$-norm in $\mathbb{R}^2$. Now, by applying the following triangle inequality
\[
\left| \|\xi\|_{\ell^2} - \|\eta\|_{\ell^2} \right| \leq \|\xi-\eta\|_{\ell^2},
\]
with $\xi= (t^{\frac{1}{2}} u(y), (1-t)^\frac{1}{2} v(y))$ and $\eta= (t^\frac{1}{2} u(x), (1-t)^\frac{1}{2} v(x))$, we obtain
\[
(\sigma_t(x)- \sigma_t(y))^2 \leq t (u(x)-u(y))^2 + (1-t) (v(x)-v(y))^2, \quad \text{for any} \ x, y \in \mathbb{R}^N.
\]
Multiplying the above inequality by the fractional kernel $\frac{1}{|x-y|^N}$ and integrating on $\Omega \times \Omega$, we obtain
\begin{equation}\label{convex-est-2}
   I_1(\sigma_t) \leq t I_1(u) + (1-t) I_2(v). 
\end{equation}
Moreover, it is easy to see that
\begin{equation}\label{convex-est-1}
    I_2(\sigma_t) = \into \l(h_{\Omega}(x) + \rho_N\r) (t u^2 + (1-t) v^2) ~dx = t I_2(u) + (1-t) I_2(v).
\end{equation}
Finally, by combining \eqref{convex-est-1} and \eqref{convex-est-2}, we obtain the required claim.
\end{proof}
\end{lemma}
\begin{theorem}\label{thm-sign-changing}
    Let $a \in L^\infty(\Omega)$ such that $a(x) \geq \lambda_{1}^L$ for a.e. in $\Omega$ and 
    \[
    |\{x \in \Omega: a(x) > \lambda_1^L\}| >0.
    \]
   Let $v$ be a non-trivial solution of 
    \begin{equation}\label{weighted-problem}
        L_\Delta v = a(x) v \ \text{in} \ \Omega, \quad \text{and} \quad v=0 \ \text{in} \ \mathbb{R}^N \setminus \Omega. 
    \end{equation}
Then, $v$ is a sign-changing function in $\Omega.$  
\end{theorem}
\begin{proof}
    We proceed by contradiction. Suppose that $v \in \h(\Omega)$ be the non-negative weak solution of the problem \eqref{weighted-problem} (if $v \leq 0$ in $\Omega$, we can repeat the proof by replacing $v$ by $-v$). 
    By \eqref{quadratic-form}, we have
\[
\mathcal{E}_L(v, \phi) = \mathcal{E}(v, \phi) - \int_{\Omega} (a(x) - \rho_N) v \phi ~dx = \int_{\Omega} (j \ast v) \phi ~dx \geq 0, \quad \text{for all} \ \phi \in \h(\Omega), \phi \geq 0.
\]
Hence, $v$ is a non-trivial non-negative weak supersolution of the equation $Iv- (a(x) -\rho_N) v=0$ in $\Omega$ in the sense of \cite{Jarohs-Weth-2018}, where $I$ is an integral operator associated with the
kernel $k$ defined in \eqref{kernels:op}. Therefore, by applying \cite[Theorem 1.1]{Jarohs-Weth-2018} yields that $v>0$ in $\Omega.$  Let $u \in \h(\Omega)$ be a solution of the minimization problem 
    \[
    \lambda_{1,L} := \min\{{\mathcal{E}_L(u,u)}: u\in \h(\Omega), \|u\|_{2}=1\}.
    \]
By using \cite[Theorems 1.4 and 1.11]{ChenWeth2019} and 
\cite[Theorem 1.9]{Dyda-Jarohs-Sk-2024}, $u >0$ in $\Omega$ and $u, v \in C(\overline{\Omega}) \cap L^\infty(\Omega).$ Moreover, by \cite[Theorem 1.4 and Corollary 5.3]{HernandezLopezSaldana2024}, there exists $M \in (0,1)$ such that
\begin{equation}\label{ratio:bounds}
    M^{-1} \geq \frac{u(x)}{v(x)} \geq M, \quad \text{for all} \ x \in \Omega.
\end{equation}
Denote 
\[
v_\eps:= \begin{cases}
    v + \eps & \ x \in \Omega,\\
    0 & \ x \in \mathbb{R}^N \setminus \Omega,
\end{cases} \quad u_\eps:= \begin{cases}
    u + \eps & \ x \in \Omega,\\
    0 & \ x \in \mathbb{R}^N \setminus \Omega,
\end{cases}
\]
and 
    \[
    \sigma_t^\eps(x) = (t u_\eps^2 + (1-t) v_\eps^2)^\frac{1}{2}.
    \]
First, we show that
\[
w_\eps:= \frac{u_\eps^2-v_\eps^2}{v_\eps} \in \h(\Omega), \quad \text{for all} \ \eps \in (0,1).
\]
It is easy to see that $w_\eps =0$ in $\mathbb{R}^N \setminus \Omega.$ For $\eps \in (0,1)$, we have
\[
\begin{split}
    |w_\eps(x) - w_\eps(y)| & \leq \left|\frac{u_\eps^2(x)}{v_\eps(x)} - \frac{u_\eps^2(y)}{v_\eps(y)}\right| + |v_\eps(x) - v_\eps(y)|\\
    & \leq \left|\frac{u_\eps^2(x)}{v_\eps(x)} - \frac{u_\eps^2(y)}{v_\eps(x)}\right| + \left|\frac{u_\eps^2(y)}{v_\eps(x)} - \frac{u_\eps^2(y)}{v_\eps(y)}\right| + |v(x) - v(y)|\\
    & \leq \frac{1}{\eps} |u_\eps^2(x)- u_\eps^2(y)| + \frac{u_\eps^2(y)}{v_\eps(x) v_\eps(y)} \left|v_\eps(x) - v_\eps(y)\right| + |v(x) - v(y)|\\
    & \leq \frac{2(\|u\|_\infty + \eps)}{\eps} |u(x)- u(y)| + \left(1 + \frac{(\|u\|_\infty + \eps)^2}{\eps^2} \right)|v(x) - v(y)| \\
    & \leq C_1(\eps, \|u\|_\infty) |u(x)- u(y)| + C_2(\eps, \|u\|_\infty) |v(x) - v(y)|
\end{split}
\]
This implies that $\mathcal{E}(w_\eps, w_\eps) < + \infty.$ Hence the claim. Moreover, by using $u, v  \in L^\infty(\Omega)$ and \eqref{ratio:bounds}, we obtain $\{w_\eps\}_{\eps \in (0,1)}$ is uniformly bounded in $L^\infty(\Omega).$ Now, by Lemma \ref{lem:convexity}, $t \mapsto \sigma_t^\eps$ is a curve of functions in $\h(\Omega)$ along which the functional $W$ is convex. This gives
    \begin{equation}\label{convex-ineq-1}
        \begin{split}
    & \mathcal{E}_L(\sigma_t^\eps, \sigma_t^\eps) - \mathcal{E}_L(v_\eps, v_\eps)  \leq t \l( \mathcal{E}_L(u_\eps, u_\eps) - \mathcal{E}_L(v_\eps, v_\eps) \r)\\
    & = t \l( \mathcal{E}_L(u, u) - \mathcal{E}_L(v, v) \r) - t c_N \intr \l(u_\eps(x) u_\eps(y) - u(x) u(y)\r) j(x-y) ~dx ~dy \\
    & \qquad + t c_N \intr \l(v_\eps(x) v_\eps(y) - v(x) v(y)\r) j(x-y) ~dx ~dy \\
    & \qquad + t \rho_N \int_{\R^N} (u_\eps^2-u^2) ~dx - t \rho_N \int_{\R^N} (v_\eps^2-v^2) ~dx\\
    & = t \l(\lambda_{1,L} \|u\|_2^2 - \int_{\Omega} a(x) v^2 ~dx \r) - t c_N \intr \l(u_\eps(x) u_\eps(y) - u(x) u(y)\r) j(x-y) ~dx ~dy \\
    & \qquad + t c_N \intr \l(v_\eps(x) v_\eps(y) - v(x) v(y)\r) j(x-y) ~dx ~dy \\
    & \qquad + t \rho_N \int_{\R^N} (u_\eps^2-u^2) ~dx - t \rho_N \int_{\R^N} (v_\eps^2-v^2) ~dx\\
    \end{split}
\end{equation}
By Lemma \ref{lem:convexity}, the function $\Phi : [0,1] \to \mathbb{R}$ defined by $\Phi(t) := W(\sigma_t)$ is convex. Hence, by using the fact $v$ satisfies \eqref{weighted-problem} and taking $w_\eps$ as a test function, we have the following estimate 
\begin{equation}\label{convex-ineq-2}
    \begin{split}
\mathcal{E}_L(\sigma_t^\eps, \sigma_t^\eps) & - \mathcal{E}_L(v_\eps, v_\eps) = \Phi(t)- \Phi(0) \geq t \Phi'(0) = t \mathcal{E}_L\l(v_\eps, \frac{u_\eps^2-v_\eps^2}{v_\eps}\r) = t \mathcal{E}_L\l(v_\eps, w_\eps\r)\\
& = t \mathcal{E}_L\l(v, w_\eps\r) - t c_N \intr (v_\eps(x)-v(x)) w_\eps(y) j(x-y) ~dx ~dy \\
& \qquad + t \rho_N \int_{\R^N} (v_\eps(x)-v(x)) w_\eps(x)~dx\\
& = t \into a(x) v w_\eps ~dx - t c_N \intr (v_\eps(x)-v(x)) w_\eps(y) j(x-y) ~dx ~dy \\
& \qquad + t \rho_N \int_{\R^N} (v_\eps(x)-v(x)) w_\eps(x)~dx.
\end{split}
\end{equation}
Combining \eqref{convex-ineq-1} and \eqref{convex-ineq-2}, we obtain
\begin{equation}\label{eigen-lower:est}
    \begin{split}
    \lambda_{1,L} \|u\|_2^2 - \int_{\Omega} a(x) v^2 ~dx & \geq \into a(x) v w_\eps ~dx - c_N \intr (v_\eps(x)-v(x)) w_\eps(y) j(x-y) ~dx ~dy \\
    & \qquad + \rho_N \int_{\R^N} (v_\eps(x)-v(x)) w_\eps(x)~dx\\
    & \qquad  +  c_N \intr \l(u_\eps(x) u_\eps(y) - u(x) u(y)\r) j(x-y) ~dx ~dy \\
    & \qquad -  c_N \intr \l(v_\eps(x) v_\eps(y) - v(x) v(y)\r) j(x-y) ~dx ~dy \\
    & \qquad -  \rho_N \int_{\R^N} (u_\eps^2-u^2) ~dx + \rho_N \int_{\R^N} (v_\eps^2-v^2) ~dx= \lambda I_0 + \sum_{i=1}^4 I_i.
\end{split}
\end{equation}
Since $v w_\eps \in L^\infty(\Omega)$ and  $v w_\eps \to u^2-v^2$ a.e. in $\Omega$, by applying Lebesgue dominated convergence theorem, we obtain
\[
I_0 \to \into (u^2-v^2) ~dx  = \|u\|_2^2 - \|v\|_2^2 = 0.
\]
By using the definition of the kernel $j$ in \eqref{kernels:op}, we have
\[
\begin{split}
    |I_1| + |I_2| \leq \l(c_N |\Omega| + \rho_N \r) \eps \|w_\eps\|_1 \to \eps \to 0.
\end{split}
\]
Moreover, since $u_\eps \to u$, $v_\eps \to v$ in $L^2(\Omega)$ and a.e. in $\Omega$, and $\supp(u_\eps), \supp(u) \subseteq \Omega$, 
\[
I_3, I_4, I_5, I_6 \to 0 \ \text{as} \ \eps \to 0.
\]
Using the above convergence estimate in \eqref{eigen-lower:est}, we obtain
\[
\int_{\Omega} (\lambda_1^L - a(x)) u^2 ~dx \geq
\lambda_{1,L} \|u\|_2^2 - \int_{\Omega} a(x) v^2 ~dx \geq \into a(x) (u^2-v^2) ~dx \Longrightarrow \int_{\Omega} (\lambda_1^L - a(x)) u^2 ~dx \geq 0,
\]
which is not possible. This contradicts our assumption that $v \geq 0$ in $\Omega.$ Therefore, the non-trivial solution $v$ of \eqref{weighted-problem} changes sign in $\Omega$.
\end{proof}
Next, as an application of the above result, we prove the sign-changing property of the eigenfunctions with eigenvalues $\lambda> \lambda_1^L.$ 
\begin{corollary}\label{pos-first-ef}
    Let $v \geq 0$ be the weak solution of the problem
    \begin{equation}\label{eigen-value-problem}
        L_\Delta v = \lambda v \ \text{in} \ \Omega, \quad \text{and} \quad v=0 \ \text{in} \ \mathbb{R}^N \setminus \Omega. 
    \end{equation}
    Then, $\lambda= \lambda_{1}^L$ where $$\lambda_{1}^L := \min\{{\mathcal{E}_L(u,u)}: u\in \h(\Omega), \|u\|_{2}=1\}.$$ 
    In other words, if $u$ is any weak solution of \eqref{eigen-value-problem} with $\lambda > \lambda_{1}^L,$ then $u$ is a sign-changing function. 
\end{corollary}
\begin{proof}
The proof follows by taking $a(x) \equiv \lambda$ in Theorem \ref{thm-sign-changing}.
\end{proof}

\section{Construction of the curve in spectrum}\label{section:curve-construction}

This section is devoted to the construction of a non-trivial curve in the spectrum $\Sigma_L$ of $L_\Delta$. For this, we fix $r \in \mathbb{R}$ and define
\[
E_r(u) = \mathcal{E}_L(u,u)-r\int_\Omega (u^+)^2 ~dx, \quad u\in \mathbb H(\Omega).
\]
Let us consider the set 
\[\mathcal{P}= \{u\in \mathbb H(\Omega):~ I(u)=\int_\Omega u^2 ~dx =1\}\]
and define $\tilde{E_r}$ as the $E_r$ restricted to $\mathcal P$.
Since $E_r$ is $C^1$ functional in $\mathbb{H}(\Omega)$, by Lagrange multiplier rule, we say that $u\in \mathcal P$ is a critical point of $\tilde E_r$ if and only if there exists a $t \in \mathbb{R}$ such that $E_r'(u)= t I'(u)$ {\it i.e.}
\[\mathcal E_L(u,v)- r\int_\Omega u^+v ~dx =t\int_\Omega uv ~dx, \quad \text{for all} ~ v \in \mathbb H(\Omega).\]
On simplifying, this means that $u\in \mathcal P$ solves $(P_{r+t,t})$ in the weak sense {\it i.e.}
\begin{equation}\label{weak-formulation}
    \mathcal E_L(u,v)= (r+t)\int_\Omega u^+v ~dx -t\int_\Omega u^-v ~dx, \quad \text{for all} ~ v\in \mathbb H(\Omega) \quad \text{and} \quad (r+t,t)\in \Sigma_L. 
\end{equation}
By taking $u=v$ in \eqref{weak-formulation}, it is easy to observe that the Lagrange multiplier $t$ is equal to the corresponding critical value $\tilde E_r(u)$. Hence we have the following result:
\begin{lemma}\label{lem1}
    For $r \in  \mathbb{R}$, $(r+t,t)\in \Sigma_L$ if and only if  there exists a $u\in \mathcal P$ which is a critical point of $\tilde E_r$ satisfying $\tilde E_r(u)=t$.
\end{lemma}
{Next, in order to identify the points in the spectrum $\Sigma_L$ of $L_\Delta$, we are interested in finding the critical points of the constrained functional $\tilde E_r$, in view of Lemma \ref{lem1}. A first critical point comes from the global minimization of the constrained functional $\tilde E_r$ because for any $u \in \mathcal{P}$, we have
\[
\tilde E_r (u) \geq \lambda_{1}^L \|u\|_2^2 - r \|u^+\|_2^2 \geq \lambda_1^L - r.
\]
where $\varphi_1$ denotes the first eigenfunction of $L_\Delta$ corresponding to the first eigenvalue $\lambda_1^L$ satisfying $\|\varphi_1\|_2=1$ {\it i.e.} $\varphi_1 \in \mathcal P$. From now on, we assume that $r\geq 0$ which is not a restriction, since $\Sigma_L$ is symmetric with respect to diagonal on $(\alpha,\beta)$ plane. The following result concern the existence of the first critical point of $\tilde E_r$.}

\begin{proposition}\label{prop2}
    $\varphi_1$ is the first critical point of $\tilde E_r$ such that
    \begin{enumerate}
        \item[(i)] $\varphi_1$ is a global minimum of $\tilde E_r$ and
        \item[(ii)] $\tilde E_r(\varphi_1)= \lambda_1^L-r$.
    \end{enumerate}
    Thus $(\lambda_1^L,\lambda_1^L-r)\in \Sigma_L$ lying on the vertical line through $(\lambda_1^L,\lambda_1^L).$
\end{proposition}
\begin{proof}
    It is easy to see that 
    \[\tilde E_r(\varphi_1) = E_r(\varphi_1)= \mathcal E_L(\varphi_1,\varphi_1)-r = \lambda_1^L-r.\]
    We recall that $\lambda_1^L = \inf\limits_{u\in \mathcal P}\mathcal E_L(u,u)$ which further gives
    \begin{align*}
        \tilde E_r(u)=  \mathcal{E}_L(u,u)-r\int_\Omega (u^+)^2 ~dx \geq \lambda_1^L - r= \tilde E_r(\varphi_1) \quad \text{for any} \ u \in \mathcal{P},
    \end{align*}
since $ \int_\Omega (u^+)^2 ~dx \leq \int_\Omega |u|^2 ~dx =1$. This finishes the proof while using Lemma \ref{lem1}.
\end{proof}
{ Next, we establish the existence of a second critical point, which arises as a strict local minimum of the constrained functional $\tilde E_r$.}
\begin{lemma}\label{lem2}
    For any $u\in \mathbb H(\Omega)$, 
    \[\mathcal{E}_L(u,u)\geq \mathcal{E}_L(u^+,u^+)+ \mathcal{E}_L(u^-,u^-).\]
\end{lemma}
\begin{proof}
Note that, we have
\[
\begin{split}
(u(x)-u(y))^2 &= ((u^+ - u^-)(x) - (u^+ - u^-)(y))^2 = ((u^+(x)-u^+(y)) - (u^-(x) - u^-(y)))^2\\
& = (u^+(x)-u^+(y))^2 + (u^-(x) - u^-(y))^2 - 2 (u^+(x)-u^+(y)) (u^-(x) - u^-(y))\\
& = (u^+(x)-u^+(y))^2 + (u^-(x) - u^-(y))^2 + 2 u^+(x) u^-(y) + 2 u^-(x)u^+(y)) \\
& \geq (u^+(x)-u^+(y))^2 + (u^-(x) - u^-(y))^2,
\end{split}
\]
and
\[
\begin{split}
u(x) u(y) & = (u^+ - u^-)(x) (u^+ - u^-)(y) \\
& = u^+(x) u^+(y) + u^-(x) u^-(y) - u^-(x) u^+(y) - u^-(y) u^+(x) \leq u^+(x) u^+(y) + u^-(x) u^-(y).
\end{split}
\]
Finally, by using the above inequalities relations in definition of $\mathcal{E}_L(u,u)$, we have the required claim. 
\end{proof}

\begin{lemma}\label{lem3}
Let $0\not \equiv v_k\in \mathbb H(\Omega)$ satisfy $v_k\geq 0$ a.e. in $\Omega$ and $|v_k>0|\to 0$ for $k \in \mathbb{N}$. Then, 
   \[
   \frac{\mathcal{E}_L(v_k,v_k)}{\|v_k\|_2^2} \to +\infty \quad \text{as} \quad k \to \infty.\]
\end{lemma}
\begin{proof}
    Denote $w_k :=\frac{v_k}{\|v_k\|_2^2}$ and assume on the contrary that $\{\mathcal{E}_L(w_k,w_k)\}$ has a bounded subsequence. { Now, by the definition of $\mathcal{E}_L(\cdot, \cdot)$, \cite[Lemma 3.4]{HernandezSaldana2022} and \eqref{compact-embed}, there exists a $w\in \mathbb H(\Omega)$ and for a further subsequence, we have
    \[w_k \rightharpoonup w ~\text{in}~\mathbb H(\Omega) \quad \text{and} \quad w_k \to w ~\text{in}~ L^2(\Omega).\]
    }
    Since $w_k\geq 0$, so $w\geq 0$ a.e. in $\Omega$ and $\int_\Omega w^2 =1$. Thus, for some $\epsilon>0$, $\sigma := |w>\epsilon|>0$ which further says that 
    \[|w_k>\frac{\epsilon}{2}|>\frac{\sigma}{2}, \quad \text{for $k$ sufficiently large.}\]
    Therefore, for sufficiently large $k$,
    \[|v_k>0|\geq |v_k >\frac{\epsilon}{2}\int_\Omega v_k^2|>\frac{\sigma}{2}\]
    which contradicts $|v_k>0|\to 0$.
    \end{proof}

\begin{proposition}\label{prop3}
     $-\varphi_1$ is the second critical point of $\tilde E_r$ such that
    \begin{enumerate}
        \item[(i)] $\varphi_1$ is a strict local minimum of $\tilde E_r$ and
        \item[(ii)] $\tilde E_r(-\varphi_1)= \lambda_1^L$.
    \end{enumerate}
    Thus, $(\lambda_1^L+r,\lambda_1^L)\in \Sigma_L$ lying on the horizontal line through $(\lambda_1^L,\lambda_1^L).$
\end{proposition}
\begin{proof}
    It is easy to verify that $\tilde E_r(-\varphi_1)= \lambda_1^L.$ To prove \textit{(i)}, on the contrary, suppose there exists a $\{u_k\}\subset \mathcal P$ such that $u_k\not\equiv -\varphi_1$ for all $k$, $\tilde E_r(u_k)\leq \lambda_1^L$ and 
    \begin{equation}\label{eq1}
        u_k\to -\varphi_1~\text{in}~\mathcal H(\Omega).
    \end{equation}
    \textbf{Claim:} $u_k$ changes sign for sufficiently large $k$.\\
    Due to \eqref{eq1} and $\varphi_1 >0$ in $\Omega$, $u_k$ must be negative for some $x\in\Omega$, when $k$ is large enough. If $u_k\leq 0$ a.e. in $\Omega$, then 
    \[\tilde E_r(u_k) = \mathcal E_L(u_k,u_k)-r\int_\Omega(u_k^+)^2 ~dx= \mathcal E_L(u_k,u_k) > \lambda_1^L,\]
    since $u_k\not \equiv \pm\varphi_1$. This contradicts $\tilde E_r(u_k)\leq \lambda_1^L$, hence the claim. Next, we define
    { \[
    t_k:= \frac{\mathcal{E}_L(u_k^+,u_k^+)}{\int_\Omega(u_k^+)^2 ~dx}
    \]
    }
    so that by using Lemma \ref{lem2}, we get
    \begin{align*}
        \tilde E_r(u_k) &\geq \mathcal{E}_L(u_k^+,u_k^+)+ \mathcal{E}_L(u_k^-,u_k^-)-r\int_\Omega(u_k^+)^2 ~dx =(t_k -r)\int_\Omega(u_k^+)^2 ~dx + \mathcal{E}_L(u_k^-,u_k^-)\\
        &\geq (t_k -r)\int_\Omega(u_k^+)^2 ~dx+ \lambda_1^L\int_\Omega(u_k^-)^2 ~dx.
    \end{align*}
Combining this with $$\tilde E_r(u_k)\leq \lambda_1^L = \lambda_1^L\int_\Omega(u_k^+)^2 ~dx + \lambda_1^L\int_\Omega(u_k^-)^2 ~dx,$$ we obtain
\[t_k -r \leq \lambda_1^L, ~\text{since}~ \int_\Omega(u_k^+)^2 ~dx>0.\]
    But \eqref{eq1} says that $|u_k>0|\to 0$ as $k\to \infty$ and hence $t_k\to \infty$, due to Lemma \ref{lem3}. This contradicts $t_k -r \leq \lambda_1^L$ which finishes the proof.
\end{proof}
{ \begin{remark}
    When $r=0$, the two critical values $\tilde E_r(\varphi_1)$ and $\tilde E_r(-\varphi_1)$ coincides as well as the critical points in $\Sigma_L.$
\end{remark}
Now we aim to obtain the third critical point, using a version of the Mountain Pass theorem, see \cite[Proposition 2.5]{{cuesta}}. In order to use it, we derive the $(P.S.)$ condition and the geometry of the functional $E_r$ on $\mathcal{P}.$ We define the norm of the derivative at $u \in \mathcal{P}$ of the restriction $\tilde{E}_r$ of $E_r$ to $\mathcal{P}$ as
\begin{equation}\label{def-deri-norm}
    \|\tilde{E}_r'(u)\|_\ast = \min\{ \|E_r'(u) - t I'(u)\|_{\mathbb{H}^\ast(\Omega)} : t \in \mathbb{R}\}
\end{equation}
where $\|\cdot\|_{\mathbb{H}^\ast(\Omega)}$ denotes the norm on the dual space $\mathbb{H}(\Omega)$. }
\begin{lemma}\label{PS-cond-critical}
    { $E_r$} satisfies the $(P.S.)$ condition on $\mathcal P$ { {\it i.e.} for any sequence $u_k \in \mathcal{P}$ such that $E_r(u_k)$ is bounded and $\|\tilde{E}_r'(u_k)\|_\ast \to 0$, $u_k$ admits a convergent subsequence.}
\end{lemma}
\begin{proof}
Let $\{u_k\}\subset \mathcal P$ and $\{t_k\} \subset \mathbb R$ be such that there exists $K>0$ for which 
\begin{equation}\label{eq3}
    { | E_r(u_k)|  =\left|\mathcal{E}_L(u_k,u_k)-r\int_\Omega (u^+_k)^2 ~dx \right|} \leq K
\end{equation}
and 
\begin{equation}\label{eq4}
{\left| \mathcal{E}_L(u_k,v)- r \int_\Omega u_k^+v ~dx - t_k \int_\Omega u_k v ~dx \right| \leq \epsilon_k \|v\|}, \quad \text{for all} ~ v \in \mathbb H(\Omega),
\end{equation}
where $\epsilon_k\to 0$. { Since $u_k \in \mathcal{P}$ for all $k \in \mathbb{N}$, \cite[Lemma 3.4]{HernandezSaldana2022} and \eqref{eq3} implies that $\{u_k\}$ is bounded in $\mathbb{H}(\Omega)$}. So upto a subsequence (denoted by same notation) and { using \eqref{compact-embed}}, there exists a $u \in \mathbb H(\Omega)$ such that $u_k\to u$ a.e. in $\Omega$ and
\begin{equation}\label{eq5}
     u_k \rightharpoonup u \quad \text{in } \mathbb H(\Omega) \quad \text{and} \quad u_k \to u \quad \text{in } L^2(\Omega). 
\end{equation}
Putting $v=u_k$ in \eqref{eq4}, we easily get that $\{t_k\}$ is bounded in $\mathbb R$. Then taking $v=u_k-u$ in \eqref{eq4}, we obtain 
\begin{align*}
    \mathcal E_L(u_k,u_k-u)= r\int_\Omega u_k^+(u_k-u) ~dx +t_k\int_\Omega u_k(u_k-u) ~dx +O(\epsilon_k)
\end{align*}
So, due to \eqref{eq5} and boundedness of $\{t_k\}$, we conclude that 
\[\lim_{k\to \infty}\mathcal E_L(u_k,u_k-u)=0.\]
Finally, by using the $\text{(S)}$-property of the operator $L_\Delta$ (see, \cite[Lemma 3.3]{Arora-Hajaiej-Perera-2026}), we obtain $u_k \to u$ in $\mathbb H(\Omega)$, which finishes the proof.
\end{proof}

The next lemma describes geometry of the constrained functional $\tilde E_r$ near the local minimum $-\varphi_1$.

\begin{lemma}\label{MP-geo}
    Let $\epsilon_0>0$ be such that
    \begin{equation}\label{local-minimum-condn}
       \tilde E_r(u) > \tilde E_r(-\varphi_1) \quad \text{for all} \ u\in B_{\epsilon_0}(-\varphi_1)\cap \mathcal P \ \text{with} \ u\neq -\varphi_1 
    \end{equation}
    where $B_{\epsilon_0}(-\varphi_1)$ denotes the ball in $\mathbb H(\Omega)$ of radius $\eps_0$ centered at $-\varphi_1$. Then, for any $\epsilon\in (0,\epsilon_0)$,
    \begin{equation}
        \inf\{\tilde E_r(u):~u\in \mathcal P,~\|u-(-\varphi_1)\|=\epsilon\}> \tilde E_r(-\varphi_1).
    \end{equation}
\end{lemma}
\begin{proof}
    On the contrary, let us assume that there exists a $\epsilon\in (0,\epsilon_0)$ for which the above infimum is equal to $\tilde E_r(-\varphi_1)=\lambda_1^L$. This means there exists a sequence $\{u_k\}\subset \mathcal P$ satisfying $\|u_k-(-\varphi_1)\|=\epsilon$ and 
    \begin{equation}\label{upper-bound}
       \tilde E_r(u_k) \leq \lambda_1^L + \frac{1}{2k^2}. 
    \end{equation}
    Now we choose a $\delta>0$ such that $0<\epsilon -\delta < \epsilon+\delta <\epsilon_0$ and define the set
    \[M_\delta = \{u\in \mathcal P: ~ \epsilon-\delta \leq \|u-(-\varphi_1)\|\leq \epsilon +\delta\}.\]
    Due to choice of $\epsilon_0$ and our contradiction hypothesis, we easily get $\inf\{\tilde E_r(u):~u\in M_\delta\}= \lambda_1^L$. We now apply the Ekeland's variational principle to $\tilde E_r$ on $\mathcal P$, for each $k$, to obtain the existence of  $v_k \in M_\delta$ satisfying
    \begin{align}
        \tilde E_r(v_k)&\leq \tilde E_r(u_k), \quad \|u_k-v_k\| \leq \frac{1}{k}\label{eq7}\\
        \tilde E_r(v_k)&\leq \tilde E_r(u) + \frac{1}{k}\|u-v_k\|, \quad \text{for all} \  u \in M_\delta. \label{eq7-1}
    \end{align}
    \textbf{Claim:} $\{v_k\}$ is a $(P.S.)$ sequence for $\tilde E_r$ on $\mathcal P$.\\
   Establishing this claim shall suffice because Lemma \ref{PS-cond-critical} will immediately give us that { for a subsequence}, $v_k\to v$ strongly in $\mathbb H(\Omega)$, for some $v\in \mathbb H(\Omega).$ Moreover due to compact embedding, $v \in \mathcal P$ and by \eqref{eq7} with $\|u_k-(-\varphi_1)\|=\epsilon$, we obtain
   \[\|v-(-\varphi_1)\|=\epsilon ~\text{and}~ \tilde E_r(v)=\lambda_1^L\]
   which is a contradiction to \eqref{local-minimum-condn}.
   
   { By \eqref{upper-bound} and \eqref{eq7}, it is easy to see that $\tilde E_r(v_k)$ is bounded. Next, we need to show that $\|\tilde E_r(v_k)\|_*\to 0$, to prove the desired claim. For this, we fix $k>\frac{1}{\delta}$ and choose $w \in \mathbb H(\Omega)$ tangent to $\mathcal{P}$ at $v_k$ satisfying
    \[\int_\Omega v_k w ~dx =0.\]
    Now we define for each $t\in \mathbb R$,
    \[f_t= \frac{v_k+tw}{\sigma(t)} \quad \text{and} \quad \sigma(t)=\|v_k+tw\|_2\]
    so that for small enough $|t|$, we get $f_t\in M_\delta$ and by using the fact that $w$ is tangent to $\mathcal{P}$ at $v_k$, we deduce that
\[
\sigma(t) \to 1, \quad \frac{1-\sigma(t)}{t} \to 0 \quad \text{and} \quad \frac{1 - \sigma(t)^2}{t} \quad \text{as} \ t \to 0.
\]
Taking $u=f_t$ in \eqref{eq7-1} and using the above limits we obtain
\begin{equation}\label{est-PS-cond-2-1}
     \begin{split}
    |\langle \tilde E_r'(v_k), w\rangle| &= \lim_{t \to 0} \frac{\tilde E_r(v_k) - \tilde E_r(v_k+tw)}{t} \\
    & \leq  \lim_{t \to 0} \left(\frac{1}{k} \frac{1}{\sigma(t)} \left\|v_k \frac{(1-\sigma(t))}{t} + w\right\| +  \left(\frac{1 - \sigma(t)^2}{t}\right) \frac{\tilde E_r(v_k +tw)}{\sigma(t)^2}\right)\\
    & \leq \frac{1}{k} \|w\|
    \end{split}
\end{equation}
for all $w \in \mathbb{H}(\Omega)$ such that $w$ is tangent to $\mathcal{P}$ at $v_k.$ Now, if $w$ is arbitrary in $\mathbb{H}(\Omega)$, we choose $\theta_k$ such that $w- \theta_k v_k$ is tangent to $\mathcal{P}$ at $v_k$ {\it i.e.}  $\theta_k = \int_{\Omega} v_k w ~dx.$ Replacing $w$ by $w-\theta_k v_k$ in \eqref{est-PS-cond-2-1}, we obtain
\[
\left|\langle \tilde E_r'(v_k), w\rangle - j_n \int_{\Omega} v_k w ~dx \right| \leq \frac{1}{k} \|w-\theta_k v_k\| \leq \eps_k \|w\|  \quad \text{where} \  j_n:= \langle \tilde E_r'(v_k), v_k \rangle \ \text{and} \ \eps_k \to 0 \ \text{as} \ k \to 0.
\]
Thus, $\|\tilde{E}_r'(v_k)\|_\ast \to 0$ in view of \eqref{def-deri-norm} and the desired claim. }
\end{proof}
By Lemmas \ref{PS-cond-critical} and \ref{MP-geo} and Propositions \ref{prop2} and \ref{prop3}, there exists a $\eps_0>0$ such that for any $\eps \in (0, \eps_0)$ such that $\|\varphi_1\|\geq \frac{\epsilon}{2}$, we have 
\begin{equation}\label{eq9}
    \inf\{\tilde E_r(u):~u\in \mathcal P, ~\|u-(-\varphi_1)\|=\epsilon\}> \max\{\tilde E_r(-\varphi_1),\tilde E_r(\varphi_1)\}.
\end{equation}
Next, we derive the existence of third critical point of the constrained energy functional $\tilde E_r.$
\begin{proposition}\label{mountain-third-critical-point}
The set 
\begin{equation}\label{def:path-set}
    \Gamma = \{\gamma \in C([-1,1], \mathcal P):~\gamma(-1)=-\varphi_1, ~\gamma(1)=\varphi_1\}
\end{equation}
is non empty and 
\begin{equation}\label{eq8}
    c(r) = \inf_{\gamma \in \Gamma } \max_{u\in \gamma [-1,+1]} E_r(u) >\lambda_1^L
\end{equation}
is a critical value of $\tilde E_r$. Moreover the same conclusion holds if \eqref{eq9} is replaced by the condition
\[\inf_{\gamma \in \Gamma } \max_{u\in \gamma [-1,+1]} E_r(u) > \max\{\tilde E_r(-\varphi_1),\tilde E_r(\varphi_1)\}.\]
\end{proposition}
\begin{proof}
    Let $\phi\in \mathbb H(\Omega)$ be such that it is not a multiple of $\varphi_1$, then consider the path 
    \[\gamma (t)= \frac{t\varphi_1+(1-t)\phi}{\|t\varphi_1+(1-t)\phi\|_2}.\]
    Clearly $\gamma \in \Gamma$ i.e. $\Gamma$ is non empty. Finally, by using \cite[Proposition 2.5]{cuesta} in view of \eqref{eq9} and \cite[Remark 2.6]{cuesta}, we conclude that $c(r)$ is a critical value of $\tilde E_r$ with
    \[c(r) > \max\{\tilde E_r(-\varphi_1),\tilde E_r(\varphi_1)\}=\lambda_1^L.\]
\end{proof}

We end this section with the following theorem which is a consequence of the Lemma \ref{lem1} and Proposition \ref{mountain-third-critical-point}:

\begin{theorem}\label{third-crit-pt}
    The points $(r+c(r),c(r))\in \Sigma_L$ for each $r\geq 0$, where $c(r)$ is defined in \eqref{eq8}.
\end{theorem}

{ This yields for $r>0$ a third critical point $(r+c(r), c(r))$ in $\Sigma_L$ on the line parallel to the diagonal passing through $(r, 0)$.} Since, $\Sigma_L$ is symmetric with respect to diagonal in the $(\alpha, \beta)$ plane, so by virtue of Theorem \ref{third-crit-pt}, we define the curve below, which belongs to $\Sigma_L$ as
\[\mathcal C := \{(r+c(r),c(r)), (c(r),r+c(r)):~r\geq 0\}. \]

\section{First nontrivial curve}\label{section-nontrivial-curve}
The aim of this section is to show that the points of $\mathcal C$ are the first non trivial points with respect to given $r\geq 0$ which lies in the intersection of $\Sigma_L$ and the line $(r,0)+t(1,1)$ for any real $t$. In order to show this, we first prove the following lemma.
\begin{lemma}\label{lemma:lower-bounds}
The following statements hold true:
\begin{enumerate}
     \item[\textnormal{(i)}] $\rho_N \geq 0$ if $N \geq 2$ and $\rho_1 = -2 \gamma \approx - 1.154$ where $\gamma = - \Gamma'(1)$ is the Euler-Mascheroni constant.
    \item[\textnormal{(ii)}] $\lambda_1^L + c_N |\Omega| \geq 0$ for all $N \geq 1.$
\end{enumerate}
\end{lemma}
\begin{proof}
    By definition of $\rho_N$ and using $\psi\left(\frac{1}{2}\right) = -\gamma - 2 \ln 2$ and $\psi(1) = -\gamma$, we have
\[
\rho_1 = 2 \ln 2 + \psi\left(\frac{1}{2}\right) - \gamma = - 2\gamma,
\]
and 
    \[
    \rho_2 = 2 \ln 2 + \psi\left(1\right) -\gamma = 2 \ln 2 - 2 \gamma >0.
    \]
Now, for $N \geq 3$, using $\psi$ is an increasing function, we have
\[
\rho_N = 2 \ln 2 + \psi\left(\frac{N}{2}\right) -\gamma \geq  2 \ln 2 + \psi(1) -\gamma = 2 \ln 2 - 2 \gamma >0 \quad \text{for} \ N \geq 3.
\]   
By \cite[Theorem 1.2 (i)]{ChenVeron2023}, we have
\[
\lambda_1^L + d_N |\Omega| \geq 0, \quad \text{where} \quad d_N:= \frac{2 \omega_{N-1}}{N^2 (2 \pi)^N}, \ N \geq 1.
\]
Now, to show (ii), it is enough to show that $c_N \geq d_N$ for all $N \geq 1.$ 
\[
c_n \geq d_N \quad \Longleftrightarrow \quad \frac{2}{\omega_{N-1}} \geq \frac{2w_{N-1}}{N^2 (2 \pi)^N} \quad \Longleftrightarrow \quad N (2 \pi)^\frac{N}{2} \geq \omega_{N-1} = \frac{2 \pi^{\frac{N}{2}}}{\Gamma\left(\frac{N}{2}\right)} \quad \Longleftrightarrow \quad 2^{\frac{N}{2}-1} N \Gamma \left(\frac{N}{2}\right) \geq 1
\]
where the last inequality holds true for every $N \geq 1.$ 

\end{proof}
\begin{theorem}\label{isolated}
    There does not exist any sequence $(\alpha_k,\beta_k) \in \Sigma_L$ with $(\alpha_k,\beta_k)\in \Sigma_L$ satisfying $\alpha_k,\beta_k>\lambda_1^L$ converges to $(\lambda_1^L,\lambda_1^L).$ This means the lines $\lambda_1^L\times \mathbb R$ and $\mathbb R \times \lambda_1^L$ are isolated in $\Sigma_L$.
\end{theorem}
\begin{proof}
    On contrary, assume that there exists a sequence $(\alpha_k,\beta_k) \in \Sigma_K$ with $\alpha_k,\beta_k>\lambda_1^L$ such that 
    \[(\alpha_k,\beta_k) \to (\alpha,\beta) \text{ with } \alpha=\lambda_1^L ~or~\beta=\lambda_1^L\]
    and $u_k$ denotes the corresponding solution of \eqref{1} with respect to $(\alpha_k,\beta_k)$, satisfying $\|u_k\|_2=1$. This gives 
    \begin{align*}
        \mathcal E_L(u_k,u_k) &= \alpha_k\int_\Omega (u_k^+)^2 ~dx + \beta_k\int_\Omega (u_k^-)^2 ~dx\\
        \implies \mathcal E(u_k,u_k) &\leq \alpha_k\int_\Omega (u_k^+)^2 ~dx+ \beta_k\int_\Omega (u_k^-)^2 ~dx + c_N \iint_{|x-y|\geq 1} \frac{u_k(x)u_k(y)}{|x-y|^N} ~dx ~dy
    \end{align*}
    from which it easily follows that $\{u_k\}$ is bounded in $\mathbb H(\Omega)$, since $\|u_k\|_2=1$. Thus $u_k\rightharpoonup u$ weakly in $\mathbb H(\Omega)$ and strongly in $L^2(\Omega)$, for some $u\in \h$. Without loss of generality, let $\alpha=\lambda_1^L$. { By using the $\text{(S)}$-property of the operator $L_\Delta$ (see, \cite[Lemma 3.3]{Arora-Hajaiej-Perera-2026}) as in Lemma \ref{PS-cond-critical}}, we deduce that $u$ weakly solves 
    \begin{equation}\label{eq10}
    L_\Delta u= \lambda_1^L u^+-\beta u^- ~\text{in}~ \Omega, ~u=0 ~\text{in}~\mathbb R^N \setminus \Omega.
    \end{equation}
    Testing the above problem with $u^+$ and using 
    \begin{equation}\label{use:ineq}
        (u(x)-u(y))(u^+(x)-u^+(y))= (u^+(x)-u^+(y))^2+u^+(x)u^-(y) +u^+(y)u^-(x)
    \end{equation}
    for any $(x,y)\in \mathbb R^2$ and symmetricity properties, we get
    \begin{align*}
        \frac{c_N}{2} &\iint_{|x-y|\leq 1}\frac{(u^+(x)-u^+(y))^2}{|x-y|^N}+c_N \iint_{|x-y|\leq 1}\frac{u^+(x)u^-(y)}{|x-y|^N} -c_N \iint_{|x-y|\geq 1} \frac{u^+(x)u^+(y)}{|x-y|^N} \\
        &\quad +c_N \iint_{|x-y|\geq 1}\frac{u^-(x)u^+(y)}{|x-y|^N}+\rho_N\int_\Omega(u^+)^2 =\lambda_1^L\int_\Omega(u^+)^2 \\
        \implies& \mathcal E_L(u^+,u^+)+ c_N\iint_{|x-y|\leq 1}\frac{u^+(x)u^-(y)}{|x-y|^N} +c_N \iint_{|x-y|\geq 1}\frac{u^-(x)u^+(y)}{|x-y|^N}=\lambda_1^L\int_\Omega(u^+)^2\\
        \implies & \mathcal E_L(u^+,u^+) \leq \lambda_1^L\int_\Omega(u^+)^2 \leq \mathcal E_L(u^+,u^+).
    \end{align*}
    Hence we have
    \[\mathcal E_L(u^+,u^+) = \lambda_1^L\int_\Omega(u^+)^2 \]
    which implies either $u^+ \equiv 0$ or $u^+ = \varphi_1$. In case $u^+\equiv 0$ {\it i.e.} $u\leq 0$, then \eqref{eq10} tells us that $u$ is a {non-positive eigenfunction of $L_\Delta$}. By Corollary \ref{pos-first-ef}, we conclude that $\beta = \lambda_1^L$ and $u=-\varphi_1$. Thus in any case 
    \[u_k \to \varphi_1 ~\text{or}~ -\varphi_1 ~\text{strongly in}~ L^2(\Omega).\]
    This implies as $k\to \infty,$
    \begin{equation}\label{eq11}
        \text{either}~ |\{x\in \Omega: ~u_k(x)<0\}| \to 0~ \text{or}~  |\{x\in \Omega: ~u_k(x)>0\}| \to 0.
    \end{equation}
    { Consider $|\{x\in \Omega: ~u_k(x)>0\}| \to 0.$ By Theorem \ref{thm-sign-changing}, $u_k$ changes sign if $(\alpha_k,\beta_k) \in \Sigma_K$ with $\alpha_k,\beta_k>\lambda_1^L$}. We argue now to contradict \eqref{eq11}. Since, $u_k$ satisfies \eqref{1} weakly, using $u_k^+$ as test function, we obtain
    \[
    \begin{split}
        \|u_k^+\| = \mathcal E(u_k^+,u_k^+) & \leq \alpha_k\int_\Omega(u_k^+)^2 ~dx + c_N \iint_{|x-y|\leq 1}\frac{u_k^+(x)u_k^+(y)}{|x-y|^N} ~dx ~dy - \min\{\rho_N, 0\} \int_\Omega u_k^2 ~dx \\
        & \leq (\alpha_k + c_N |\Omega| - \min\{\rho_N, 0\})  \int_\Omega(u_k^+)^2.
    \end{split}    
    \]
By Lemma \ref{lemma:lower-bounds}, we have $\lambda_1^L + c_N |\Omega| - \min\{\rho_N, 0\} >0$. Now, by choosing $k$ large enough such that $d_k:= (\alpha_k + c_N |\Omega| - \min\{\rho_N, 0\}) >0$ and by applying H\"older inequality, we obtain
\begin{equation}\label{lower:est-1}
       \|u_k^+\|^2 = \mathcal E(u_k^+,u_k^+) \leq (\alpha_k + c_N |\Omega| - \min\{\rho_N, 0\})  \int_\Omega(u_k^+)^2 \leq 2 d_k \|(u_k^+)^2\|_{\psi,\Omega} \ \|{\bf 1}_{\{u_k(x) >0\}}\|_{\psi^\ast,\Omega}
    \end{equation}  
where $\psi:[0, \infty) \to [0, \infty)$ such that 
\begin{equation}\label{modu-condi-1}
    \psi(t) :=  t \ln(e+t^\frac{1}{2}), 
\quad \psi(t^2) \leq \phi(t):=t^2 \ln(e+t) \ \text{for all} \ t \geq 0 \quad \text{and} \quad { \frac{\psi^{-1}(t)}{t} \to 0 \ \text{as} \ t \to +\infty},
\end{equation}
$\psi^\ast$ is a conjugate function of $\psi$, and $\|\cdot\|_{\ph,\Omega}$ represents the Luxemburg norm defined by (see \cite[Definition 3.2.1]{Harjulehto-Hasto-2019})
	\[
		\|u\|_{\ph, \Omega} = \inf \bigg\{ \lambda > 0 \,:\, \varrho_\ph \l( \frac{u}{\lambda} \r)  \leq 1 \bigg\}.
\]
\textbf{Claim 1:} $\|(u_k^+)^2\|_{\psi,\Omega} \leq \|u_k^+\|_{\phi, \Omega}^2 \leq C \|u_k^+\|^2$ {for some $C>0.$}
\newline
Denote 
\[ v_k:= \frac{u_k^+}{\|u_k^+\|_{\phi, \Omega}} \quad \text{such that} \quad \|v_k\|_{\phi, \Omega} =1.
\]
By using the unit ball property of the norm (see \cite[Lemma 3.2.3]{Harjulehto-Hasto-2019}), \eqref{modu-condi-1} and \cite[Theorem 3.6]{Arora-Giacomoni-Vaishnavi}, we have
\[
\int_{\Omega} \phi(v_k) ~dx \leq 1 \quad \Longrightarrow \quad \int_{\Omega} \psi(v_k^2) ~dx \leq 1 \quad \Longrightarrow \quad \|v_k^2\|_{\psi, \Omega} \leq 1 \quad \Longrightarrow \quad \|(u_k^+)^2\|_{\psi,\Omega} \leq \|u_k^+\|_{\phi, \Omega}^2 \leq { C\|u_k^+\|^2}.
\]
\textbf{Claim 2:} $\|{\bf 1}_{\{u_k(x) >0\}}\|_{\psi^\ast,\Omega} = \frac{1}{(\psi^\ast)^{-1}\left(|\{u_k(x) >0\}|^{-1}\right)} \leq C |\{u_k(x) >0\}| \ \psi^{-1}\left(|\{u_k(x) >0\}|^{-1}\right) $
\newline
By the definition of Luxemburg norm and \cite[Proposition 2.4.9 and Theorem 2.4.8]{Harjulehto-Hasto-2019}, we obtain
\[
\begin{split}
    \|{\bf 1}_{\{u_k(x) >0\}}\|_{\psi^\ast,\Omega} &= \inf \left\{\lambda: \int_\Omega \psi^\ast\left(\frac{{\bf 1}_{\{u_k(x) >0\}}}{\lambda}\right) ~dx \leq 1\right\} = \inf\left\{ \lambda: \psi^\ast\left( \frac{1}{\lambda}\right) |\{u_k(x) >0\}| \leq 1 \right\}\\
    & = \inf\left\{ \lambda: \psi^\ast \left(\frac{1}{\lambda}\right)  \leq \frac{1}{|\{u_k(x) >0\}|}\right\} = \inf\left\{ \lambda: \frac{1}{\lambda}  \leq (\psi^\ast)^{-1}\left(|\{u_k(x) >0\}|^{-1}\right)\right\}\\
    & = \frac{1}{(\psi^\ast)^{-1}\left(|\{u_k(x) >0\}|^{-1}\right)} 
    \leq C_1 |\{u_k(x) >0\}| \ \psi^{-1}\left(|\{u_k(x) >0\}|^{-1}\right)
\end{split}
\]
{for some $C_1>0.$} Now, by using \eqref{eq11}, \textbf{Claim 1} and \textbf{Claim 2} in \eqref{lower:est-1}, { \eqref{modu-condi-1}}, the definition of $\psi$ and the fact that $d_k>0$ is bounded above, we obtain
\[
1 \leq 2C_0 d_k |\{u_k(x) >0\}| \ \psi^{-1}\left(|\{u_k(x) >0\}|^{-1}\right) \to 0 \ \text{as} \ k \to \infty
\]
which is a contradiction. { By the repeating the same arguments as above for the case $|\{x\in \Omega: ~u_k(x)<0\}| \to 0$ by testing the equation \eqref{1} with $u_k^-$, we will obtain a contradiction. Hence, the claim.}
\end{proof}
By analogy with Lemma 3.5 in \cite{cuesta}, we now formulate the following topological result concerning $\mathcal{P}$.
\begin{lemma}\label{topo-1}
   The following holds concerning $\mathcal P$-
   \begin{enumerate}
       \item[(i)] $\mathcal P$ is locally arcwise connected.
       \item[(ii)] Any open connected subset $\mathcal O$ of $\mathcal P$ is arcwise connected.
       \item[(iii)] If $\mathcal O'$ is a connected component of an open subset $\mathcal O$ of $\mathcal P$ then $\partial \mathcal O' \cap \mathcal{O} = \emptyset$,
   \end{enumerate}
\end{lemma}
Our next result can also be done along the lines of proof of Lemma 3.6 of \cite{cuesta} by replacing $\|.\|_{1,p}$ with $\|.\|$, $\tilde J_s$ with $\tilde E_r$ and defining $\mathcal O$ as
\[\mathcal O := \{u\in \mathcal P:~\tilde E_r(u)<\alpha\}\]
for some real $\alpha$.

\begin{lemma}\label{topo-2}
    Any connected component of $\mathcal O$ contains a critical point of $\tilde E_r$.
    \end{lemma}

We now finish this section by proving the main result of this section below.

\begin{theorem}\label{mainthrm-1}
    For any $r\geq 0$, the point $(r+c(r), c(r))$ is the first non trivial point of $\Sigma_L$ which lies in the intersection of $\Sigma_L$ and the line parallel to the diagonal through $(r,0)$.
\end{theorem}
\begin{proof}
    On contrary, let $(r+\theta, \theta)$ lies in $\Sigma_L$ for $\theta \in (\lambda_1^L, c(r))$ and $\theta$ is minimum quantity satisfying this. This choice of $\theta$ is possible due to Theorem \ref{isolated} and {$\Sigma_L$ being closed}. Our aim is to construct a path in $\Gamma$ on which $\tilde E_r(\cdot) \leq \theta$ which shall be a contradiction, owing to the definition of $c(r)$ in \eqref{eq8}.

    We first observe that our contrary hypothesis tells us that $\tilde E_r$ has a critical value $\theta \in (\lambda_1^L, c(r)))$ and $\tilde E_r$ has no critical value in $(\lambda_1^L, \theta)$. Let $u\in \mathcal P$ be the critical point of $\tilde E_r$ {\it w.r.t.} critical value $\theta$. Then, $u$ changes sign by {Theorem \ref{thm-sign-changing}} and  solves 
    \[\mathcal E_L(u,v)= (r+\theta)\int_\Omega u^+v-\theta\int_\Omega u^-v, \quad \text{for all} ~ v \in \mathbb H(\Omega).\]
Taking $v=u^+$ and $v=u^-$, we get 
    \begin{align}
        \mathcal{E}_L(u,u^+) &= (r+\theta)\int_\Omega (u^+)^2 ~dx \label{mt1-1}\\
         \mathcal{E}_L(u,u^-) &= -\theta\int_\Omega (u^-)^2 ~dx.\label{mt1-2}
    \end{align}
    On simplifying them, we obtain
    \begin{align}\label{est:+-}
         \mathcal E_L(u^+,u^+)+ h(u^+,u^-) =(r+\theta)\int_\Omega(u^+)^2 ~dx \quad \text{and} \quad \mathcal E_L(u^-,u^-)+ h(u^+,u^-) =\theta\int_\Omega(u^-)^2 ~dx.
    \end{align}
where
    \[h(u^+,u^-) =  c_N\iint_{|x-y|\leq 1}\frac{u^+(x)u^-(y)}{|x-y|^N} ~dx ~dy +c_N \iint_{|x-y|\geq 1}\frac{u^-(x)u^+(y)}{|x-y|^N} ~dx ~dy \geq 0,\]
Moreover, we get
    \begin{align*}
        \tilde E_r(u)& =\theta\\
        \tilde E_r\left(\frac{u^+}{\|u^+\|_{2}}\right) &= \theta - \frac{h(u^+,u^-)}{\|u^+\|_{2}^2} = \tilde E_r\left(\frac{-u^-}{\|u^-\|_{2}^2}\right)\leq \theta \\
        \tilde E_r\left(\frac{u^-}{\|u^-\|_{2}}\right) &= \theta -r- \frac{h(u^+,u^-)}{\|u^-\|_{2}^2}\leq \theta -r.
    \end{align*}
Next, we consider three paths in $\mathcal{P}$ which go respectively from $u$ to $\frac{u^+}{\|u^+\|_{2}^2}$, from $\frac{u^+}{\|u^+\|_{2}^2}$ to $\frac{u^-}{\|u^-\|_{2}^2}$, and from $\frac{-u^-}{\|u^-\|_{2}^2}$ to $u$ via:
\begin{align*}
        & \gamma_1(t)= \frac{tu^+ + (1-t)u}{\|tu^+ + (1-t)u\|_2}= \frac{u^+-(1-t)u^-}{\|u^+ - (1-t)u^-\|_2}, \quad \gamma_1(0) = u, \quad \gamma_1(1) = \frac{u^+}{\|u^+\|_{2}^2},\\
        & \gamma_2(t)= \frac{u^+ - tu}{\|t u^-+(1-t)u^+\|_2} = \frac{t u^-+(1-t)u^+}{\|t u^-+(1-t)u^+\|_2},\quad \gamma_2(0)= \frac{u^+}{\|u^+\|_{2}^2}, \quad \gamma_2(1) = \frac{u^-}{\|u^-\|_{2}^2}, \\
        &\gamma_3(t) = \frac{t u -(1-t) u^-}{\|t u -(1-t) u^-\|_2} =\frac{tu^+-u^-}{\|tu^+-u^-\|_2} , \quad \gamma_3(0) = \frac{-u^-}{\|u^-\|_2} , \quad \gamma_3(1) = u.
    \end{align*}
Next, we examine the levels of $\tilde E_r$ along all the paths $\gamma_i$ for $i=1,2$ and $3.$ separately below. \newline
\textbf{Estimates for path $[0,1] \ni t \mapsto \gamma_1(t):$}
\begin{equation}\label{est:path-3}
    \begin{split}
  \tilde E_r(\gamma_1(t)) \|tu^+ + (1-t)u\|_2^2 &= \mathcal{E}\l(u^+-(1-t)u^-,u^+-(1-t)u^-\r) \\
  & \quad - c_N \iint_{|x-y|\geq 1} \frac{(u^+-(1-t)u^-)(x) (u^+-(1-t)u^-)(y)}{|x-y|^N}  ~dx ~dy \\
  & \quad + \rho_N \int_{\Omega} (u^+-(1-t)u^-)^2 ~dx - r \int_\Omega (u^+-(1-t)u^-)^2 ~dx. 
\end{split}
\end{equation}
Note that
\[
\begin{split}
    & \l[(u^+-(1-t)u^-)(x) - (u^+-(1-t)u^-)(y)\r]^2 = \l[(u^+(x) - u^+(y)) - (1-t)(u^-(x) - u^-(y)) \r]^2\\
    & \qquad = \l[u^+(x) - u^+(y)\r]^2 + (1-t)^2\l[u^-(x) - u^-(y)\r]^2 +{ 2  (1-t) \l[u^+(x) u^-(y) + u^-(x) u^+(y)\r]},
\end{split}
\]
\[
\begin{split}
  (u^+-(1-t)u^-)(x) &(u^+-(1-t)u^-)(y) \\
  &= u^+(x) u^+(y) + (1-t)^2u^-(x) u^-(y) -(1-t) \l[ u^+(x) u^-(y) + u^-(x) u^+(y)\r],
\end{split}
\]
and
\[
\begin{split}
(u^+-(1-t)u^-)^2(x) = (u^+)^2(x) + (1-t)^2(u^-)^2(x) 
\end{split}
\]
Using the above estimates in \eqref{est:path-3}, \eqref{est:+-} and $h(u^+, u^-) \geq 0$ leads to
{\[
\begin{split}
    \tilde E_r(\gamma_1(t)) \|tu^+ + (1-t)u\|_2^2 & =  \mathcal{E}_L(u^+, u^+) + (1-t)^2\mathcal{E}_L(u^-, u^-) + 2 (1-t) h(u^+, u^-) \\
    & \qquad - r (1-t)^2 \int_{\Omega} (u^+)^2 ~dx - r\int_\Omega (u^-)^2 ~dx \\
    & =  (r+\theta)\int_\Omega(u^+)^2 ~dx + \theta (1-t)^2 \int_\Omega(u^-)^2 ~dx - h(u^+,u^-)
- (1-t)^2 h(u^+,u^-) \\
 & \qquad + 2 (1-t) h(u^+, u^-) - r \int_{\Omega} (u^+)^2 ~dx - r  (1-t)^2\int_\Omega (u^-)^2 ~dx \\
 & ={ \theta \|tu^+ + (1-t)u\|_2^2 -r(1-t)^2 \|u^-\|_2^2 - t^2 h(u^+, u^-).}
\end{split}
\]
which further gives
\[
\tilde E_r(\gamma_1(t)) \leq \theta \quad \text{for all} \ t \in [0,1].
\]}
\textbf{Estimates for path $[0,1] \ni t \mapsto \gamma_2(t):$}
\begin{equation}\label{est:path-1}
    \begin{split}
  \tilde E_r(\gamma_2(t)) \|tu^- + (1-t)u^+\|_2 &= \mathcal{E}\l(tu^- + (1-t)u^+, tu^- + (1-t)u^+\r) \\
  & \quad - c_N \iint_{|x-y|\geq 1} \frac{(tu^- + (1-t)u^+)(x) (tu^- + (1-t)u^+)(y)}{|x-y|^N}  ~dx ~dy \\
  & \quad + \rho_N \int_{\Omega} (tu^- + (1-t)u^+)^2 ~dx - r \int_\Omega (tu^- + (1-t)u^+)^2 ~dx. 
\end{split}
\end{equation}
Note that
\[
\begin{split}
    & \l[(tu^- + (1-t)u^+)(x) - (tu^- + (1-t)u^+)(y)\r]^2 = \l[(1-t)(u^+(x) - u^+(y)) + t (u^-(x) - u^-(y)) \r]^2\\
    & \quad = (1-t)^2 \l[u^+(x) - u^+(y)\r]^2 + t^2 \l[u^-(x) - u^-(y)\r]^2 - 2 t(1-t) \l[u^+(x) u^-(y) + u^-(x) u^+(y)\r],
\end{split}
\]
\[
\begin{split}
  (tu^- + (1-t)u^+)(x) &(tu^- + (1-t)u^+)(y) \\
  & = (1-t)^2 u^+(x) u^+(y) + t^2 u^-(x) u^-(y) + t(1-t) \l[ u^+(x) u^-(y) + u^-(x) u^+(y)\r],
\end{split}
\]
and
\[
\begin{split}
(tu^- + (1-t)u^+)^2(x) = (1-t)^2 (u^+)^2(x) + t^2 (u^-)^2(x) 
\end{split}
\]
Using the above estimates in \eqref{est:path-1} and \eqref{est:+-} leads to
\[
\begin{split}
    \tilde E_r(\gamma_2(t)) \|tu^- + (1-t)u^+\|_2^2 & = (1-t)^2 \mathcal{E}_L(u^+, u^+) + t^2 \mathcal{E}_L(u^-, u^-) - 2 t (1-t) h(u^+, u^-) \\
    & \quad  - r t^2 \int_{\Omega} (u^-)^2 ~dx - r (1-t)^2 \int_{\Omega} (u^+)^2 ~dx \\
    & = \theta \|tu^- + (1-t)u^+\|_2^2 - h(u^+, u^-) - r t^2 \int_\Omega (u^-)^2 ~dx.\end{split}
\]
Now, by using the fact that $h(u^+, u^-) \geq 0$ and $r \geq 0$, we obtain
\[
\tilde E_r(\gamma_2(t)) \leq \theta \quad \text{for all} \ t \in [0,1].
\]
\textbf{Estimates for path $[0,1] \ni t \mapsto \gamma_3(t):$}
\begin{equation}\label{est:path-2}
    \begin{split}
  \tilde E_r(\gamma_3(t)) \|tu^+ -  u^-\|_2 &= \mathcal{E}\l(tu^+ - u^-,t u^+ -  u^-\r) - c_N \iint_{|x-y|\geq 1} \frac{(tu^+ -  u^-)(x) (tu^+ -  u^-)(y)}{|x-y|^N}  ~dx ~dy \\
  & \quad + \rho_N \int_{\Omega} (tu^+ -  u^-)^2 ~dx - r \int_\Omega (tu^+ -  u^-)^2 ~dx. 
\end{split}
\end{equation}
Note that
\[
\begin{split}
    & \l[(tu^+ -  u^-)(x) - (tu^+ -  u^-)(y)\r]^2 = \l[t(u^+(x) - u^+(y)) - (u^-(x) - u^-(y)) \r]^2\\
    & \qquad = t^2\l[u^+(x) - u^+(y)\r]^2 +  \l[u^-(x) - u^-(y)\r]^2 + 2 t \l[u^+(x) u^-(y) + u^-(x) u^+(y)\r],
\end{split}
\]
\[
\begin{split}
  (tu^+ - u^-)(x) &(tu^+ - u^-)(y) = t^2 u^+(x) u^+(y) +  u^-(x) u^-(y) -t \l[ u^+(x) u^-(y) + u^-(x) u^+(y)\r],
\end{split}
\]
and
\[
\begin{split}
(tu^+ -  u^-)^2(x) =t^2 (u^+)^2(x) +  (u^-)^2(x) 
\end{split}
\]
Using the above estimates in \eqref{est:path-2}, \eqref{est:+-} and $h(u^+, u^-) \geq 0$ leads to
\[
\begin{split}
    \tilde E_r(\gamma_3(t)) \|tu^+ -  u^-\|_2^2 & = t^2\mathcal{E}_L(u^+, u^+) +\mathcal{E}_L(u^-, u^-) + 2 t h(u^+, u^-) - r t^2 \int_{\Omega} (u^+)^2 ~dx\\
    & = {\theta \|tu^+ + (1-t)u^-\|_2^2 - (1-t)^2 h(u^+, u^-) -r\|u^-\|_2^2.}
\end{split}
\]
which further gives
\[
\tilde E_r(\gamma_3(t)) \leq \theta \quad \text{for all} \ t \in [0,1].
\]
Using the above paths, we can go from $u$ to $\frac{u^-}{\|u^-\|_2^2}$ (via $\gamma_1(t)$ and $\gamma_2(t)$) by staying at levels $\leq \theta$. Now we have to study the levels below $\theta-r$, so we define the set
\[\mathcal O = \{v\in \mathcal P:~ \tilde E_r(v)<\theta -r\}.\]

Due to our assumption, Propositions \ref{prop2} and \ref{prop3}, we state that $\varphi_1\in \mathcal O$ while $-\varphi_1\in \mathcal O$ if $\lambda_1^L< \theta -r$. Also, due to the choice of $\theta$ and Lemma \ref{lem1}, $\pm\varphi_1$ are the only possible critical points of $\tilde E_r$ in $\mathcal O$. Clearly,  so $\frac{u^-}{\|u^-\|_2^2}$ is a regular point of $\tilde E_r$, hence there exist $\eps>0$ and a $C^1$ path 
\[f:[-\epsilon,+\epsilon] \to S\;\;\text{with}\;\;f(0)= \frac{u^-}{\|u^-\|_2^2},\quad \frac{d}{dt}(\tilde E_r(f(t)))|_{t=0}\neq 0.\]
Following this path in either positive or negative directions will lead us to some $w\in \mathcal P$ where $\tilde E_r(w)<\theta-r$ i.e. $w\in \mathcal O$. Applying Lemma \ref{topo-1} and \ref{topo-2} to the component of $\mathcal O$ containing $w$ and keeping in mind our hypothesis, we conclude that either $\varphi_1$ or $-\varphi_1$ lies in this component. Without loss of generality, we assume $\varphi_1$ is in this component which implies that we can move from $\frac{u^-}{\|u^-\|_2^2}$ to $w$ and then $w$ to $+\varphi_1$ through a path in $\mathcal P$ while staying at levels $< \theta-r$. So, we have constructed a path in $\mathcal P$ which connects $u$ to $\varphi_1$ and we call the part of this path as $\gamma_4(t)$ which connects $\frac{u^-}{\|u^-\|_2^2}$ to $+\varphi_1$ whose level stays $\leq \theta -r$. Then, its symmetric path $-\gamma_4(t)$ connects $\frac{-u^-}{\|u^-\|_2^2}$ to $-\varphi_1$  and satisfies
\[|\tilde E_r(\gamma_4(t))-\tilde E_r(-\gamma_4(t))|\leq r\]
which gives 
\[\tilde E_r(-\gamma_4(t))\leq \tilde E_r(\gamma_4(t)) +r\leq (\theta-r)+r=\theta.\]
Altogether we have constructed a path in $\mathcal P$ from $u$ to $\frac{u^-}{\|u^-\|_2^2}$ and $\frac{u^-}{\|u^-\|_2^2}$ to $+\varphi_1$ staying at levels $\leq \theta -r$ and $-\gamma_4(1-t)$ allows us to go from  $-\varphi_1$ to $-\frac{u^-}{\|u^-\|_2^2}$ by staying at levels $\leq \theta$. Finally $-\gamma_3(t)$ brings us back from $-\frac{u^-}{\|u^-\|_2^2}$ to $u$ by staying at level $\theta$. By combining the above constructed paths, we have actually constructed a path in $\mathcal P$ from $-\varphi_1$ to $+\varphi_1$ staying at levels $\leq \theta$ which finishes the proof of this theorem.
\end{proof}

{ As an application of Theorem \ref{mainthrm-1}, we next provide the variational characterization of the second eigenvalue  $\lambda_2^L$. Moreover, the curve $\mathcal{C}$ passes through $(\lambda_2^L, \lambda_2^L)$ at $r=0.$ 
\begin{corollary}
    One has
    \[
    \lambda_2^L:= \inf_{\gamma \in \Gamma} \max_{u \in \gamma[-1,1]} \mathcal{E}_L(u,u)
    \]
    where $\Gamma$ is the family of all continuous paths in $\mathcal{P}$ going from $- \varphi_1$ to $\varphi_1.$
\end{corollary}}

\section{Properties of the non-trivial curve $\mathcal{C}$}\label{section-curve-prop}
In this section, we study some monotonicity and regularity properties of the first non-trivial curve $\mathcal{C}$ as well as its asymptotic behaviour.

\begin{proposition}\label{curve-prop-1}
   The curve $r \mapsto c(r)$ for $r\in \mathbb R^+$ is non-increasing and Lipschitz continuous.
\end{proposition}
\begin{proof}
    The proof follows by adopting the same arguments as in \cite[Proposition 4.1]{cuesta}. For the sake of completeness, we give the details. Let $r < r'.$ Since, $\tilde{E}_r(u) \geq \tilde{E}_{r'}(u)$ for any $u \in \mathcal{P}$, we have $c(r) \geq c(r').$ Now let $\eps>0.$ Then, there exists a path $\gamma \in \Gamma$ such that
    \[
    \max_{u \in \gamma[-1,1]} \tilde{E}_{r'} (u) \leq c(r') +\eps.
    \]
    Therefore, we have
    \[
    0 \leq c(r) - c(r') \leq \max_{u \in \gamma[-1,1]} \tilde{E}_{r} (u) - \max_{u \in \gamma[-1,1]} \tilde{E}_{r'} (u) + \eps
    \]
Denoting $w_0$ be a point in $\gamma[-1,1]$ where $\tilde{E}_r$ achieves its maximum on $\gamma[-1,1],$ we have
\[
0 \leq c(r) - c(r') \leq \tilde{E}_{r} (w_0) - \tilde{E}_{r'} (w_0) + \eps \leq (r'-r) + \eps.
\]
Since $\eps>0$ is arbitrary, we obtain the Lipschitz property of the given map.
\end{proof}
\begin{lemma}\label{lem-eigenvalues-order}
    Let $A$, $B$ be two bounded open sets in $\mathbb{R}^N$, with $A \subset B$ and $B$ is connected then $\lambda_1^L(A) > \lambda_1^L(B).$
\end{lemma}
\begin{proof}
    The proof follows from the variational characterization of the first eigenvalue $\lambda_1^L$.
\end{proof}
\begin{lemma}\label{lem-perturbed-fucik}
    Let $(\alpha, \beta) \in \mathcal{C}$ and let $a_1, a_2 \in L^\infty(\Omega)$ such that
    \begin{equation}\label{pertur-cond-1}
        \lambda_1^L \leq a_1(x) \leq \alpha \quad \text{and} \quad \lambda_1^L \leq a_2(x) \leq \beta \ \text{a.e. in} \ \Omega. 
    \end{equation}
Assume that
\begin{equation}\label{pertur-cond-2}
    \lambda_1^L < a_1(x) \quad \text{and} \quad \lambda_1^L < a_2(x) \quad \text{on subsets of positive measure}.
\end{equation}
Then, any non-trivial solution $u$ of 
\begin{equation}\label{pertubed-prob}
    \left\{\begin{aligned}
   L_\Delta u\: &= a_1(x) u^+ - a_2(x) u^- &&~~\text{in} ~~ \Omega, \\
      u&=0 &&~~\text{in} ~~\mathbb R^N\setminus \Omega,
    \end{aligned} \right.\tag{$P_{a_1, a_2}$}
\end{equation}
changes sign in $\Omega$, and 
\[
a_1(x) = \alpha \ \text{a.e. on} \ \{x \in \Omega: u(x) >0\}, \quad a_2(x) = \beta \ \text{a.e. on} \ \{x \in \Omega: u(x) <0\}
\]
(and consequently $u$ is an eigenfunction associated to the point $(\alpha, \beta)$ of $\mathcal{C}$).
\end{lemma}
\begin{proof}
   Let $u$ be a nontrivial solution of \eqref{pertubed-prob}. Replacing $u$ by $-u$ if necessary, we may assume that the point $(\alpha, \beta) \in \mathcal{C}$ satisfies $\alpha \geq \beta.$ We first show that $u$ changes sign in $\Omega.$ Arguing by contradiction, suppose this is not the case; without loss of generality, assume that $u \geq 0$ a.e. in $\Omega$ (the opposite case can be treated analogously). Then $u$ satisfies 
   \[
   L_{\Delta} u = a_1(x) u \quad \text{in} \ \Omega \quad \text{and} \quad u =0 \quad \text{in} \ \mathbb{R}^N \setminus \Omega.
   \]
Since $a_1$ satisfies \eqref{pertur-cond-1} and \eqref{pertur-cond-2}, by Theorem \ref{thm-sign-changing}, $u$ is a sign-changing function, which is a contradiction. Therefore, $u$ changes sign in $\Omega.$ Now, we assume by contradiction that either 
\begin{equation}\label{claim1}
    |\{x \in \Omega : a_1(x) < \alpha \ \text{and} \ u(x) >0\}| >0
\end{equation}
or
\begin{equation}\label{claim2}
|\{x \in \Omega: a_2(x) < \beta \ \text{and} \ u(x)<0\}| >0.
\end{equation}
Here, as before, $|\cdot|$ denotes Lebesgue measure. Suppose \eqref{claim1} holds true (a similar argument would work for \eqref{claim2}). Put $\alpha-\beta = r \geq 0.$ Then, $\beta = c(r)$ where $c(r)$ is given by \eqref{eq8}. We will show that there exists a path $\gamma \in \Gamma$ such that 
\begin{equation}\label{path-contra}
    \max_{u \in \gamma[-1,1]} \tilde{E}_r(u) < \beta,
\end{equation}
which yields a contradiction with the definition of $c(r).$ In order to construct a path $\gamma$, we show that
\begin{equation}\label{eigen:upper:bound}
    \frac{\mathcal{E}_L(u^+, u^+)}{\|u^+\|_2^2} < \alpha \quad \text{and} \quad \frac{\mathcal{E}_L(u^-, u^-)}{\|u^-\|_2^2} < \beta.
\end{equation}
By taking $u^+$ and $-u^-$ as test functions in \eqref{pertubed-prob}, we obtain
\begin{equation}
    \mathcal{E}_L(u^+, u^+) \leq \mathcal{E}_L(u^+, u^+) - \mathcal{E}_L(u^-, u^+)= \mathcal{E}_L(u, u^+) = \into a_1(x) (u^+)^2 ~dx < \alpha \|u^+\|_2^2
\end{equation}
and \begin{equation}
    \mathcal{E}_L(u^-, u^-) \leq \mathcal{E}_L(u^+, - u^-) + \mathcal{E}_L(-u^-, -u^-)= \mathcal{E}_L(u, -u^-) = \into a_2(x) (u^-)^2 ~dx < \beta \|u^+\|_2^2
\end{equation}
Since $\mathcal{E}_L(u^-, u^+) \leq 0$, we obtain the claim in \eqref{eigen:upper:bound}. 
Moreover, we have
\[
\begin{split}
\tilde{E}_r \left(\frac{u}{\|u\|_2}\right) & = \frac{\mathcal{E}_L(u,u)}{\|u\|_2^2} - r \into \frac{(u^+)^2}{\|u\|_2^2} = \frac{\mathcal{E}_L(u^+,u^+)}{\|u\|_2^2} + \frac{\mathcal{E}_L(u^-,u^-)}{\|u\|_2^2} - 2 \frac{h(u^+, u^-)}{\|u\|_2^2} - r \into \frac{(u^+)^2}{\|u\|_2^2} ~dx\\
& \leq \frac{\mathcal{E}_L(u^+,u^+)}{\|u\|_2^2} + \frac{\mathcal{E}_L(u^-,u^-)}{\|u\|_2^2} - r \into \frac{(u^+)^2}{\|u\|_2^2} ~dx\\
& \leq (\alpha - r) \into \frac{(u^+)^2}{\|u\|_2^2} ~dx + \beta \into \frac{(u^-)^2}{\|u\|_2^2} ~dx = \beta.
\end{split}
\]
where
    \[h(u^+,u^-) =  c_N\iint_{|x-y|\leq 1}\frac{u^+(x)u^-(y)}{|x-y|^N} +c_N \iint_{|x-y|\geq 1}\frac{u^-(x)u^+(y)}{|x-y|^N} \geq 0.\]
Similarly, we have
\[
\tilde{E}_r\left(\frac{u^+}{\|u^+\|_2}\right) < \alpha-r =\beta \quad \text{and} \  \tilde{E}_r\left(\frac{u^-}{\|u^-\|_2}\right) < \beta-r.
\]

Now, by using Lemma \ref{topo-2}, we have that there exists a critical point of $\tilde{E}_r$ in the connected component of the set $O = \{u \in \mathcal{P} : \tilde{E}_r(u) < \beta - r\}$. As the point $(\alpha, \beta) \in \mathcal{C}$, the only possible critical point of $\tilde{E}_r$ is $\varphi_1$, then we can construct a path from $\varphi_1$ to $- \varphi_1$ exactly in the same manner as in Theorem \ref{mainthrm-1} satisfying \eqref{path-contra}, and hence the result follows.
\end{proof}

\begin{corollary}\label{coro-perturbed-fucik}
    Let $(\alpha, \beta) \in \mathcal{C}$ and $a_1, a_2 \in L^\infty(\Omega)$ satisfying \eqref{pertur-cond-1} and \eqref{pertur-cond-2}. If either $a_1(x) < \alpha$ a.e. in $\Omega$ or $a_2(x) < \beta$ a.e. in $\Omega$, then \eqref{pertubed-prob} has only the trivial solution.
\end{corollary}
\begin{proof}
   This follows directly from Lemma \ref{lem-perturbed-fucik}.
\end{proof}
\begin{proposition}\label{curve-prop-2}
   The curve $r \mapsto (r+c(r), c(r))$ for $r\in \mathbb R^+$ is  continuous and strictly decreasing in the sense that $r < r'$ implies $r+c(r) < r' + c(r')$ and $c(r) > c(r')$.
\end{proposition}
\begin{proof}
    The continuity of the map $r \mapsto (r+c(r), c(r))$ follows from Proposition \ref{curve-prop-1}. To prove that, the curve is strictly decreasing, we will use Corollary \ref{coro-perturbed-fucik}. Let $r < r'$. Assume by contradiction that
    \[
    \text{either} \quad r + c(r) \geq r' + c(r') \quad \text{or} \quad c(r) \leq c(r').
    \]
    In the first case, we have 
    \[
    r + c(r) \geq r' + c(r') > r + c(r') \quad \Longrightarrow \quad c(r) > c(r').
    \]
    Now, by taking $(\alpha, \beta) = (r+ c(r), c(r))$ in Corollary \ref{coro-perturbed-fucik}, we obtain the problem 
    \[
    L_\Delta v = (r'+ c(r')) v^+ - c(r') v^- ~~\text{in} ~~ \Omega, \quad v =0 \ \text{in} \ \mathbb{R}^N \setminus \Omega
    \]
    has only the trivial solution, which gives a contradiction to the fact that $(r' + c(r'), c(r')) \in \Sigma_L.$ On the other hand, if $c(r) \leq c(r')$ implies $ r + c(r) < r' + c(r').$ Now, again by applying Corollary \ref{coro-perturbed-fucik} with $(\alpha, \beta) = (r'+ c(r'), c(r'))$, we obtain the problem 
    \[
    L_\Delta v = (r+ c(r)) v^+ - c(r) v^- ~~\text{in} ~~ \Omega, \quad v =0 \ \text{in} \ \mathbb{R}^N \setminus \Omega
    \]
    has only the trivial solution, which gives a contradiction to the fact that $(r + c(r), c(r)) \in \Sigma_L.$
\end{proof}
\begin{proposition}\label{curve-prop-3}
    The limit of $c(r)$ as $r \to +\infty$ is $\lambda_1^L.$
\end{proposition}
\begin{proof}
    Assume by contradiction that there exists a $\delta>0$ such that $\max_{u \in \gamma[-1,1]} \tilde{E}_r(u) \geq \lambda_1^L + \delta$ for all $\gamma \in \Gamma$ and all $s \geq 0.$ Let $\varphi \in \mathbb H(\Omega)$ which is unbounded from above in the neighborhood of some $x_1 \in \Omega$ such that there does not exist a $r \in \mathbb{R}$ such that $\varphi \leq r \varphi_1$ a.e. in $\Omega$ and consider the path $\gamma \in \Gamma$ defined by
    \[
    \gamma(t) = \frac{t \varphi_1 + (1-|t|)\varphi}{\|t \varphi_1 + (1-|t|)\varphi\|_2}, \quad t \in [-1,1].
    \]
    The maximum of $\tilde{E}_r$ on $\gamma[-1,1]$ is achieved at say $t_r.$ Putting $v_{t_r} = t_r \varphi_1 + (1-|t_r|) \varphi$, we thus have 
    \begin{equation}\label{est-limit-1}
        \mathcal{E}_L(v_{t_r}, v_{t_r}) - r \int_\Omega (v_{t_r}^+)^2 ~dx \geq (\lambda_1^L + \delta) \int_{\Omega} |t_r \varphi_1 + (1-|t_r|) \varphi|^2 ~dx \quad \text{for all} \ r \geq 0.
    \end{equation}
    Letting $r \to + \infty$, we can assume, for a subsequence, $t_r \to t_0 \in [-1,1].$ Since $v_{t_r}$ remains bounded in $\mathbb{H}(\Omega) \cap L^2(\Omega)$ as $r \to +\infty$, it follows from \eqref{est-limit-1} that $\int_\Omega (v_{t_r}^+)^2 ~dx \to 0.$ Consequently, by compact embedding of $\mathbb{H}(\Omega)$ in $L^2(\Omega)$, we obtain
    \[
    \int_{\Omega} ((t_0 \varphi_1 + (1-|t_0|) \varphi)^+)^2 ~dx=0,
    \]
    which is impossible by the choice of $\varphi$ unless $t_0 = -1.$ So, $t_r \to -1.$ Finally, by passing limits in \eqref{est-limit-1}, we arrive at
    \[
    \lambda_1^L \int_{\Omega} |\varphi_1|^2 ~dx = \mathcal{E}_L(\varphi_1, \varphi_1) \geq (\lambda_1^L + \delta) \int_{\Omega} |\varphi_1|^2 ~dx
    \]
    which is a contradiction.
\end{proof}
\section{Nonresonance between $(\lambda_1^L, \lambda_1^L)$ and $\mathcal{C}$}\label{section:nonresonance}
In this section, we study the following problem
\begin{equation}\label{Non-resonance-prob}
    \left\{\begin{aligned}
   L_\Delta u\: &= f(x,u) &&~~\text{in} ~~ \Omega, \\
      u&=0 &&~~\text{in} ~~\mathbb R^N\setminus \Omega,
    \end{aligned} \right.\tag{$P_\Delta$}
\end{equation}
where $\frac{f(x,u)}{u}$ asymptotically lies between $(\lambda_1^L, \lambda_1^L)$ and $(\alpha, \beta) \in \mathcal{C}.$ Let $f: \Omega \times \mathbb{R} \to \mathbb{R}$ be a Caratheodory function. Given a point $(\alpha, \beta) \in \mathcal{C}$, we assume the following
\begin{equation}\label{assumpt-f-1}
    \gamma_{\pm}(x) \leq \liminf\limits_{s \to \pm \infty} \frac{f(x,s)}{s} \leq \limsup\limits_{s \to \pm \infty} \frac{f(x,s)}{s} \leq \Gamma_{\pm}(x)
\end{equation}
hold uniformly with respect to $x$, where $\gamma_{\pm}$ and $\Gamma_{\pm}$ are bounded functions which satisfy 
\begin{equation}\label{assumpt-f-2}
\begin{cases}
    \lambda_1^L \leq \gamma_+(x) < \Gamma_+(x) \leq \alpha \ \text{a.e. in} \ \Omega& \\
    \lambda_1^L \leq \gamma_-(x) < \Gamma_-(x) \leq \beta \ \text{a.e. in} \ \Omega. &
    \end{cases}
\end{equation}
We also assume the following 
\begin{equation}\label{assumpt-f-3}
    \delta_{\pm}(x) \leq \liminf\limits_{s \to \pm \infty} \frac{2F(x,s)}{s^2} \leq \limsup\limits_{s \to \pm \infty} \frac{2F(x,s)}{s^2} \leq \Delta_{\pm}(x), \quad \text{where} \quad F(x,s) = \int_0^s f(x,t) ~dt
\end{equation}
hold uniformly with respect to $x$, where  $\delta_{\pm}(x)$ and $\Delta_{\pm}(x)$ are bounded functions which satisfy
\begin{equation}\label{assumpt-f-4}
    \begin{cases}
    \lambda_1^L \leq \delta_{+}(x) \leq \Delta_{+}(x) \leq \alpha \ \text{a.e. in} \ \Omega&\\
    \lambda_1^L \leq \delta_{-}(x) \leq \Delta_{-}(x) \leq \beta \ \text{a.e. in} \ \Omega&\\
    \lambda_1^L < \delta_+ \ \text{and} \ \lambda_1^L < \delta_- \ \text{on subsets of positive measure}&\\
    \text{either} \ \Delta_{+}(x) < \alpha \ \text{a.e. in} \ \Omega \ \text{or} \ \Delta_-(x) < \beta \ \text{a.e. in} \ \Omega.
    \end{cases}
\end{equation}
Define the energy functional $\Psi : \mathbb{H}(\Omega) \to \mathbb{R}$ as
\[
\Psi(u) = \frac{\mathcal{E}_L(u,u)}{2} - \int_\Omega F(x,u) ~dx.
\]
Then, $\Psi$ is a $C^1$ functional on $\mathbb{H}(\Omega)$ and 
\[
\left\langle \Psi'(u), v \right\rangle = \mathcal{E}_L(u,v) - \int_{\Omega} f(x,u) v ~dx
\]
and the critical points of $\Psi$ are exactly the weak solution of \eqref{Non-resonance-prob}.
\begin{lemma}\label{PS-condn-reso}
    The functional $\Psi$ satisfies the $(PS)$ condition in $\mathbb{H}(\Omega).$
\end{lemma}
\begin{proof}
    Let $\{u_k\}$ be a Palais Smale sequence (in short (PS) sequence) 
    \begin{equation}\label{PS-cond}
       |\Psi(u_k)| \leq c, \quad \text{and} \ (\Psi'(u_k), \phi ) \leq \eps_k \|\phi\|  
    \end{equation}
where $c>0$ and $\eps_k \to 0$ as $k \to \infty.$ It is enough to show that the Palais Smale sequence is bounded. Assume by contradiction that $\{u_k\}$ is not a bounded sequence. Define $v_k = \frac{u_k}{\|u_k\|}$, a bounded sequence in $\mathbb{H}(\Omega).$ Then, there exists a bounded subsequence of $\{v_k\}$ (denoted by same notation) and a $v_0 \in \mathbb{H}(\Omega)$ such that $v_k \rightharpoonup v_0$ weakly in $\mathbb{H}(\Omega)$ and $v_k \to v_0$ in $L^2(\Omega)$ and $v_k \to v_0$ a.e. in $\Omega.$ Now, by using \eqref{assumpt-f-1}-\eqref{assumpt-f-2}, we have $\frac{f(x,u_k)}{\|u_k\|} \rightharpoonup f_0(x)$ in $L^2(\Omega).$ By taking $\phi= v_k - v_0$, dividing by $\|u_k\|$ and using the $(S)$ property of the operator $L_{\Delta}$ (see, \cite[Lemma 3.3]{Arora-Hajaiej-Perera-2026}), we obtain $v_k \to v_0$ in $\mathbb{H}(\Omega).$ In particular, $\|v_0\| =1$ and $v_0 \not \equiv 0.$ This further gives in view of \eqref{PS-cond}
\[
\mathcal{E}_L(v_0, \phi) - \into f_0(x) \phi ~dx = 0 \quad \forall \ \phi \in \mathbb{H}(\Omega).
\]
Now, by standard arguments based on assumption \eqref{assumpt-f-1}, $f_0(x) = a_1(x) v_0^+ - a_2(x) v_0^-$ for some bounded functions $a_1$ and $a_2$ satisfying \eqref{pertur-cond-1}. In the expression of $f_0(x)$, the value of $a_1(x)$ (respectively $a_2(x)$) on $\{x : v_0(x) \leq 0\}$ (respectively $\{x : v_0(x) \geq 0\}$) are irrelevant, and consequently we can assume that
\begin{equation}\label{PS-est-1}
    a_1(x) > \lambda_1^L \quad \text{on} \ \{x: v_0(x) \leq 0\} \quad \text{and} \quad a_2(x) > \lambda_1^L \quad \text{on} \ \{x: v_0(x) \geq 0\}. 
\end{equation}
It then follows from Lemma \ref{lem-perturbed-fucik} that either $(i): a_1(x)=\lambda_1^L$ a.e. in $\Omega$, or $(ii):  a_2(x)=\lambda_1^L$ a.e. in $\Omega$, or $(iii): v_0$
is an eigenfunction associated to the point $(\alpha, \beta)$ of $\mathcal{C}$. We will see that each case leads to a contradiction. If $(i)$ holds then by \eqref{PS-est-1}, $v_0>0$ a.e. in $\Omega$ and \eqref{pertubed-prob} implies
\begin{equation}
    \mathcal{E}_L(v_0, v_0) = \lambda_1^L \into v_0^2 ~dx.
\end{equation}
Now, dividing \eqref{PS-cond} by $\|u_k\|^2$, taking limits as $k \to \infty$ and using \eqref{assumpt-f-3}, we obtain
\[
\lambda_1^L \into v_0^2 ~dx  = \mathcal{E}_L(v_0, v_0) = \lim_{k \to \infty} \into \frac{2 F(x, u_k)}{\|u_k\|^2} ~dx \geq \into \delta_+(x) v_0^2 ~dx
\]
which is a contradiction to \eqref{assumpt-f-4}. The case $(ii)$ can be treated similarly. Now, if $(iii)$ holds, we deduce from \eqref{assumpt-f-3} that
\[
\into \left(\alpha (v_0^+)^2 + \beta (v_0^-)^2\right) ~dx = \mathcal{E}_L(v_0, v_0) = \lim_{k \to \infty} \into \frac{2 F(x, u_k)}{\|u_k\|^2} ~dx \leq \into \Delta_+(x) v_0^2 + \Delta_-(x) v_0^2 ~dx
\]
which again contradicts the assumption \eqref{assumpt-f-4}, since $v_0$ changes sign in $\Omega$ by Lemma \eqref{lem-perturbed-fucik}. Hence, $\{u_k\}$ is a bounded sequence in $\mathbb{H}(\Omega).$ 
\end{proof}
Next, we study the mountain pass geometry of the energy functional $\Psi.$
\begin{lemma}\label{lem:mpg}
    There exists a $R>0$ such that
    \begin{equation}\label{geometry}
        \max\{\Psi(R \varphi_1), \Psi(-R \varphi_1)\} \leq \max_{u \in \gamma[-1,1]} \Psi(u)
    \end{equation}
    for any $\gamma \in \Gamma_R := \{\gamma \in C([-1,1], \mathbb{H}(\Omega)): \gamma(1) = R \varphi_1, \gamma(-1) = -R \varphi_1\}.$
\end{lemma}
\begin{proof}
In view of \eqref{pertubed-prob}, first we consider the following functional associated to the functions $\Delta_{\pm}$ given by
\[
J(u) = \mathcal{E}_L(u,u) - \into \Delta_+(x) u^2 ~dx - \into \Delta_-(x) u^2 ~dx 
\]
and claim that 
\begin{equation}\label{claim-1}
    d:= \inf_{\gamma \in \Gamma} \max_{u \in \gamma[-1,1]} J(u) > 0
\end{equation}
where $\Gamma$ is defined in \eqref{def:path-set}. Denote $r = \alpha-\beta \geq 0$. Since $(\alpha, \beta) \in \mathcal{C}$, we have for any $\gamma \in \Gamma$, 
\[
\max_{u \in \gamma[-1,1]} \tilde{{E}}_r(u) \geq c(r) = \beta \quad \Longrightarrow \quad \max_{u \in \gamma[-1,1]} \left(\mathcal{E}_L(u,u) - \alpha \into (u^+)^2 ~dx - \beta \into (u^-)^2 ~dx \right) \geq 0.
\] 
which further implies $\max\limits_{u \in \gamma[-1,1]} J(u) > 0$ due to \eqref{assumpt-f-4}. Therefore, $d \geq 0.$ On the other hand, since $\delta_{\pm}(x) \leq \Delta_{\pm}(x)$ a.e. in $\Omega$, we have
\[
J(\pm\varphi) \leq \into (\lambda_1^L - \delta_{\pm}(x)) \varphi^2 ~dx < 0. 
\]
Thus, we have a mountain pass geometry for the restriction $\tilde{J}$ of $J$ to $\mathcal{P}$,
\[
\max\{\tilde{J}(\varphi_1), \tilde{J}(-\varphi_1)\} < 0 \leq \max_{u \in \gamma[-1,1]} \tilde{J}(u)
\]
for any path $\gamma \in \Gamma$ and one verifies as in Lemma \ref{PS-cond-critical} that $J$ satisfies the $(PS)$ condition on $\mathcal{P}$. Then, by using the mountain pass theorem as in Proposition \ref{mountain-third-critical-point}, we obtain $d$ is the critical value of $\tilde{J}$, {\it i.e.} there exists a $u \in \mathcal{P}$ and $\mu \in \mathbb{R}$ such that
\[
J(u) = d \quad \text{and} \quad (J'(u), \phi) = \mu (I'(u), \phi) \ \text{for all} \ \phi \in \mathbb{H}(\Omega). 
\]
Assume by contradiction that $d=0.$ Taking $\phi =u$ above, one deduces that $\mu=0$, so that $u$ is a nontrivial solution of 
\[
L_\Delta u = \Delta_+(x) u^+ - \Delta_-(x) u^- \ \text{in} \ \Omega \quad \text{and} \quad u=0 \ \text{in} \ \mathbb{R}^N \setminus \Omega.
\]
Now, by using \eqref{assumpt-f-4} and Lemma \ref{lem-perturbed-fucik}, we get a contradiction and hence the claim in \eqref{claim-1}. By \eqref{assumpt-f-3} and for any $\zeta>0$ there exists $a_\zeta \in L^1(\Omega)$ such that for a.e. $x$,
    \begin{equation}\label{perturbed-nonl}
        \begin{cases}
             (\delta_+(x) -\zeta) \frac{s^2}{2} - a_\zeta(x) \leq F(x,s) \leq (\Delta_+(x) + \zeta) \frac{s^2}{2} + a_\zeta(x), & \quad s>0\\
             (\delta_-(x) -\zeta) \frac{s^2}{p} - a_\zeta(x) \leq F(x,s) \leq (\Delta_-(x) + \zeta) \frac{s^2}{2} + a_\zeta(x), & \quad s<0.
        \end{cases}
    \end{equation}
    By the left inequalities in \eqref{perturbed-nonl} it follows that for any $R>0$ and $\zeta>0$,
\[
\Psi(\pm R \varphi_1) \leq \frac{R^2}{2} \into (\lambda_1^L-\delta_{\pm}(x)) \varphi_1^2 ~dx + \frac{\zeta R^2}{2} + \|a_\zeta\|_1  
\]
which further implies, by using \eqref{assumpt-f-4} and choosing $\zeta$ sufficiently small, that $\Psi(\pm R \varphi_1) \to -\infty$ as $R \to \infty.$ Fix $\zeta$ such that $0< \zeta < d$ and choose $R= R(\zeta)$ such that 
\begin{equation}\label{upperbound-mpg}
    \Psi(\pm R \varphi_1) \leq - \|a_\zeta\|_1
\end{equation}
where $a_\zeta$ is associated to $\zeta$ through \eqref{perturbed-nonl}. Now, let us consider a path in $\gamma \in \Gamma_R$. If $0 \in \gamma[-1,1]$, then the claim in \eqref{geometry} follows from \eqref{upperbound-mpg} and using $\Psi(0)=0$. On the other hand, if $0 \not \in \gamma[-1,1]$, then then we can consider the normalized path $\tilde{\gamma}(t)= \frac{\gamma(t)}{\|\gamma(t)\|}$, which belongs to $\Gamma.$ Since, by \eqref{perturbed-nonl}, 
\[
\Psi(u) \geq \frac{J(u) - \zeta \|u\|_2^2}{2} - \|a_\zeta\|_1
\]
we obtain
\[
\max_{u \in \gamma[-1,1]} \frac{2 \Psi(u) + 2 \|a_\zeta\|_1 + \zeta \|u\|_2^2}{\|u\|_2^2}  \geq \max_{v \in \tilde{\gamma}[-1,1]} J(v) \geq d
\]
and consequently, by the choice of $\zeta$ and \eqref{upperbound-mpg},
\[
\max_{u \in \gamma[-1,1]} \frac{2 \Psi(u) + 2 \|a_\zeta\|_1}{\|u\|_2^2}  \geq d -\eps >0 \quad \Longrightarrow \quad \max_{u \in \gamma[-1,1]} \Psi(u) > - \|a_\zeta\|_1 \geq \Psi(\pm R \varphi_1). 
\]
Hence, the required claim.
\end{proof}
\begin{theorem}
    Let \eqref{assumpt-f-1}-\eqref{assumpt-f-4} hold and $(\alpha, \beta) \in \mathcal{C}.$ Then the problem \eqref{Non-resonance-prob} admits atleast one solution $u$ in $\mathbb{H}(\Omega).$  
\end{theorem}
\begin{proof}
    The proof follows by the application of Mountain pass theorem and combining the claims in Lemmas \ref{PS-condn-reso} and \ref{lem:mpg}.
\end{proof}
\section*{Acknowledgment} The first author acknowledges financial support from the Anusandhan National Research Foundation (ANRF), India, under Grant No. ANRF/ARGM/2025/000272/MTR. This work was initiated during a research visit to the Department of Mathematics, IIT Jodhpur, and the first author acknowledges the kind hospitality received there in May 2025.

\subsection*{Ethical Approval}
Not applicable.
\subsection*{Competing interests}
The authors declare that they have no competing interests.
\subsection*{Authors' contributions}
The authors contributed equally to this work.
\subsection*{Availability of data and materials}
Data sharing is not applicable to this article as no new data were created or analyzed in this study.

\end{document}